\documentclass[preprint,12pt]{elsarticle}

\usepackage{amsmath}
\usepackage{amsfonts}
\usepackage{amsthm}
\usepackage{hyperref}
\usepackage{subfig}
\usepackage{algorithmic}
\usepackage{algorithm}
\usepackage{booktabs}
\usepackage{tikz}
\usetikzlibrary{shapes.geometric, arrows.meta, positioning, calc, backgrounds}

\newtheorem{theorem}{Theorem}
\theoremstyle{definition}
\newtheorem{example}{Example}
\newtheorem{remark}{Remark}

\begin{document}

\begin{frontmatter}

\title{\texorpdfstring{Energy Dissipation Preserving Feature-based DNN Galerkin Methods for Gradient Flows}{Energy Dissipation Preserving Feature-based DNN Galerkin Methods for Gradient Flows}}

\author[inst4,inst7]{Tao Tang}

\author[inst1,inst3,inst5,inst6]{Jiang Yang\corref{cor1}}
%\ead{yangj7@sustech.edu.cn}
%\cortext[cor1]{Corresponding author. Emails: ttang@nfu.edu.cn (T. Tang), 12131241@mail.sustech.edu.cn (Y. Zhao), zhuqh@nus.edu.sg (Q. Zhu)}
\cortext[cor1]{Corresponding author. \\ 
{\it Email addresses:} ttang@nfu.edu.cn (T. Tang), yangj7@sustech.edu.cn (J. Yang),
 12131241@mail.sustech.edu.cn (Y. Zhao), zhuqh@nus.edu.sg (Q. Zhu)}

\author[inst1]{Yuxiang Zhao}

\author[inst2]{Quanhui Zhu}

\address[inst1]{Department of Mathematics, Southern University of Science and Technology, Shenzhen, China.}

\address[inst2]{Department of Mathematics, National University of Singapore, Singapore.}

\address[inst3]{SUSTech International Center for Mathematics, Shenzhen, China.}

\address[inst4]{School of Mathematics and Statistics, Guangzhou Nanfang College, Guangzhou, China.}

\address[inst5]{Guangdong Provincial Key Laboratory of Computational Science and Material Design, SUSTech, Shenzhen, China.}

\address[inst6]{National Center for Applied Mathematics Shenzhen (NCAMS), Shenzhen, China.}

\address[inst7]{Zhuhai SimArk Technology Co., LTD, Zhuhai, China.}

\begin{abstract}
In recent years, deep learning methods, exemplified by Physics-Informed Neural Networks (PINNs), have been widely applied to the numerical solution of differential equations. However, these methods may suffer from limited accuracy, high training costs, and lack of robustness, particularly their inability to preserve the intrinsic physical structures of continuous PDE models, such as the energy dissipation property in gradient flow systems. To address these challenges, we propose a feature-based Deep Neural Network Galerkin (DNN-G) framework designed for structure-preserving simulations of gradient flows. Instead of treating neural networks merely as optimization-driven solvers, we employ them as adaptive feature generators that define nonlinear trial spaces within a Galerkin projection formulation.This formulation guarantees semi-discrete energy dissipation and can be naturally combined with energy stable time integration schemes. Several strategies for constructing neural basis functions are investigated, including random features, structured initialization, and problem-informed pre-training. Numerical experiments demonstrate that the proposed method preserves robust energy stability in high-dimensional settings and accurately captures complex topological transitions. With equivalent degrees of freedom, the DNN-G framework achieves higher accuracy than classical spectral methods, highlighting the effectiveness of neural feature representations for the numerical solution of partial differential equations.
\end{abstract}


\end{frontmatter}

%% 
%SECTION 1
\section{Introduction}\label{sec::introduction}
The numerical simulation of gradient flows plays a fundamental role in material science,
 fluid mechanics, and biological modeling. 
A broad class of dissipative systems can be formulated in the form
\begin{equation}
    \frac{\partial u}{\partial t} 
    = \mathcal{G} \frac{\delta E}{\delta u},
    \quad x \in \Omega, \quad t \in [0,T],
    \label{eq::gradient_flow}
\end{equation}
where $E[u]$ is an energy functional and 
$\mathcal{G}$ is a negative semi-definite operator characterizing the dissipation mechanism. 
This formulation immediately implies the energy dissipation law
\begin{equation}
    \frac{\mathrm{d}E}{\mathrm{d}t} = \int_{\Omega} \frac{\delta E}{\delta u} \cdot \mathcal{G} 
    \frac{\delta E}{\delta u} \mathrm{d}x \leq 0,
\end{equation} 
which governs the long-time dynamics of the system. 
Preserving this structure at the discrete level is therefore 
essential for physical fidelity and numerical stability.

Over the past decades, extensive efforts in numerical analysis have 
led to a variety of energy-stable schemes, including convex splitting methods 
\cite{shen_secondorderconvexsplitting_2012,chen_linearenergystable_2012},
scalar auxiliary variable approaches \cite{shen_scalarauxiliaryvariable_2018,shen_newclassefficient_2019}, invariant energy quadratization methods \cite{yang_numericalapproximationsthreecomponent_2017}, 
stabilized implicit-explicit schemes \cite{beylkin_newclasstime_1998,fu_energydiminishingimplicitexplicit_2024}, 
and exponential time differencing methods 
\cite{cao_exponentialtimedifferencing_2024,tian_exponentialtimedifferencing_2020,fu_energydecreasingexponentialtime_2022,wang_efficientstableexponential_2016,fu_higherorderenergydecreasingexponential_2024}. 
These methods provide rigorous energy stability and high-order temporal accuracy.
Nevertheless, most of them rely on mesh-based discretizations and 
their efficiency may deteriorate in high-dimensional problems or in situations involving 
complex geometries and topological changes.

Recently, deep neural networks (DNNs) have emerged as a flexible 
mesh-free framework for solving partial differential equations. 
Representative approaches include Physics-Informed Neural Networks (PINNs) 
\cite{raissi_physicsinformedneuralnetworks_2019} and related neural PDE solvers 
\cite{sirignano_dgmdeeplearning_2018,zang_weakadversarialnetworks_2020,e_deepritzmethod_2018,lu_learningnonlinearoperators_2021},
that approximate the solution or the solution operator using a neural network. 
Their expressive power and mesh-free nature make them attractive 
for high-dimensional problems and complex geometries 
\cite{han_solvinghighdimensionalpartial_2018,xiong_deepfiniteelement_2025,grekas_deepritzfinite_2025,bruna_neuralgalerkinscheme_2022,liu_multiscaledeepneural_2020}.
Extensions to gradient flow problems have also been explored 
\cite{yang_localdeeplearning_2022,wight_solvingallencahncahnhilliard_2021,wang_meshfreemethodinterface_2020}.

Despite their flexibility, ensuring physical consistency in neural PDE solvers 
remains a significant challenge. 
Most existing approaches enforce the governing equations through 
penalty terms in a loss function, leading to a soft-constraint formulation 
\cite{kutuk_energydissipationpreserving_2025}. 
For dissipative systems such as gradient flows, this strategy generally does not 
guarantee the discrete energy dissipation law. 
As a consequence, numerical solutions may exhibit artificial energy growth 
or oscillations, particularly in stiff regimes or long-time simulations. 
Developing neural solvers that rigorously preserve the intrinsic 
variational structure therefore remains an open problem.

Recent efforts to bridge this gap have explored diverse strategies. 
Hu et al.~\cite{hu_energeticvariationalneural_2024} proposed an energetic variational 
neural network in which the energy functional is directly embedded 
into a time-stepping optimization procedure.  
Zhang et al.~\cite{zhang_energydissipativeevolutionarydeep_2024}
combined operator learning with the SAV methodology 
to ensure the dissipation of a modified energy functional. 
While promising, these approaches typically rely on optimization 
at each time step or specialized network architectures. 
A general framework that combines neural representations with 
structure-preserving numerical discretizations remains largely unexplored.

To address these challenges, we propose a structure-preserving 
deep neural network Galerkin (DNN-G) framework for gradient flows. 
Instead of treating neural networks as optimization-based solvers, 
we interpret the outputs of the last hidden layer as adaptive 
basis functions that define a nonlinear trial space.
By projecting the gradient flow equation onto this neural trial space, 
we obtain a semi-discrete dynamical system for the coefficients. 
We show that this projected system inherits the energy dissipation law 
of the original continuous problem, thereby providing a rigorous 
structure-preserving guarantee at the semi-discrete level.

The main contributions of this work are summarized as follows:
\begin{itemize}
    \item \textbf{Structure-preserving DNN-Galerkin formulation.} 
    We establish a Galerkin projection framework based on neural basis functions 
    that guarantees semi-discrete energy dissipation for gradient flows.
The resulting semi-discrete system can be combined naturally with 
high-order time integration schemes.
\item \textbf{Adaptive neural basis construction.} 
We propose strategies for constructing and updating neural trial spaces 
during the evolution process. 
The adaptive mechanism allows the basis functions to track sharp interfaces, 
multiscale coarsening dynamics, and topological changes.
\item \textbf{Scalability and high-order accuracy.}
Numerical experiments demonstrate that the proposed method achieves 
higher accuracy under the same degrees of freedom compared with classical 
spectral discretizations, while maintaining high-order temporal accuracy 
and good scalability in higher-dimensional problems.
\end{itemize}
The remainder of this paper is organized as follows. 
Section 2 presents the theoretical formulation of the DNN-Galerkin method 
and establishes its energy-dissipative property. 
Section 3 discusses the construction and adaptation of neural basis functions. 
Section 4 provides numerical experiments demonstrating accuracy, stability, 
and scalability. 
Conclusions are drawn in Section 5.

%SECTION 2
\section{Structure-Preserving DNN Galerkin Scheme}\label{sec::DNN-G}

In this section, we develop a structure-preserving numerical framework 
for gradient flows based on neural Galerkin discretization. 
Instead of treating neural networks as direct solvers for the PDE, 
we interpret the outputs of the last hidden layer as adaptive basis functions 
that define a nonlinear trial space. 
Projecting the gradient flow equation onto this neural trial space 
yields a semi-discrete dynamical system for the coefficients. 
Unlike residual-minimization approaches such as Physics-Informed Neural Networks 
or Deep Galerkin Methods, the resulting formulation naturally preserves 
the energy dissipation structure of the underlying variational problem.

\subsection{Neural basis functions}
We employ a standard feedforward neural network to generate 
parameterized basis functions:
\begin{equation}
    \begin{aligned}
        z_0 &= I(x),\\
        z_l &= \sigma(W_l z_{l-1} + b_l), \quad l=1,\dots,L,
    \end{aligned}
\end{equation}
where $I(x)$ denotes the input feature mapping, 
$W_l$ and $b_l$ are weight matrices and bias vectors, 
and $\sigma$ is the activation function. 
Let
\begin{equation}
    \Phi(x;\theta)
=
\{\phi_1(x;\theta),\dots,\phi_m(x;\theta)\}
=
z_L
\end{equation}
denote the vector of neural features extracted from the last hidden layer \cite{yang_$e$rankstaircasephenomenon_2025}
where $\theta = \{W_l,b_l\}_{l=1}^L$ collects all trainable parameters.

When the parameter $\theta$ is fixed, $\Phi(x;\theta)$ defines a finite-dimensional trial 
space 
\begin{equation}
    \mathcal{V}_\theta = \text{span} \{\phi_1(x; \theta), \dots, \phi_m(x; \theta)\}.
\end{equation}
The trial space is generally nonlinear in the parameter $\theta$ 
but linear in the coefficients $\beta(t)$, 
which allows the projected system to retain a linear mass matrix structure.

To incorporate the initial condition exactly, 
we augment the basis by including
\begin{equation}
    \phi_0(x)=u_0(x), 
\qquad 
\phi_{m+1}(x)=1,
\end{equation}
and define the augmented basis \begin{equation}
    \Phi_a(x;\theta)
=
[u_0(x), \phi_1(x;\theta), \dots, \phi_m(x;\theta), 1]^T.
\end{equation}

The semi-discrete solution is represented as
\begin{equation}
u_h(x,t)
=
\sum_{i=0}^{m+1} \beta_i(t)\phi_i(x;\theta)
=
\beta(t)^T \Phi_a(x;\theta),
\label{eq::NN}
\end{equation}
where $\beta(t)\in\mathbb{R}^{m+2}$ are time-dependent coefficients.
Choosing $\beta(0)=[1,0,\dots,0]^T$
ensures $u_h(x,0)=u_0(x)$ exactly.

For periodic domains $\Omega$ with period $L$, 
periodicity is enforced through an input feature embedding
\begin{equation}
I(x)
=
\left(
\cos\left(\frac{2\pi x}{L}\right),
\sin\left(\frac{2\pi x}{L}\right)
\right),
\end{equation}
so that all neural basis functions are periodic by construction.

\subsection{Semi-discrete DNN-Galerkin System}

To facilitate the analysis of the energy dissipation property,
we introduce the auxiliary variable $v$ and the gradient flow \eqref{eq::gradient_flow} can then be written
in the mixed form
\begin{equation}
\begin{cases}
\dfrac{\partial u}{\partial t} = \mathcal G v, \\
v = \dfrac{\delta E}{\delta u}.
\end{cases}
\label{eq::mixed_form}
\end{equation}

We approximate both $u$ and $v$ in the neural trial space
$\mathcal V_\theta$. The semi-discrete solution for $u$ and $v$ is given by
\begin{equation}
u_h(x,t)=\sum_{i=0}^{m+1}\beta_i(t)\phi_i(x;\theta),\quad v_h(x,t)=\sum_{i=0}^{m+1}\gamma_i(t)\phi_i(x;\theta).
\end{equation}

The mixed Galerkin formulation reads:
Find $u_h,v_h\in\mathcal V_\theta$ such that
\begin{equation}
    \begin{aligned}
\int_\Omega \partial_t u_h \psi \mathrm{d}x &=\int_\Omega \mathcal G v_h \psi\mathrm{d}x,
\quad \forall \psi\in\mathcal V_\theta,
\\
\int_\Omega v_h \psi\mathrm{d}x &=\int_\Omega \frac{\delta E}{\delta u}(u_h)\psi\mathrm{d}x,
\quad \forall \psi\in\mathcal V_\theta.
    \end{aligned}
    \label{eq::mix_weakform}
\end{equation}

Substituting the neural expansions leads to the
finite-dimensional system
\begin{equation}
    \begin{aligned}
        \mathbf M(\theta)\dot{\beta}&=\mathbf G(\theta)\gamma,
\\
\mathbf M(\theta)\gamma &=\mathbf r(\beta),
    \end{aligned}
    \label{eq::ode_system}
\end{equation}
where
\begin{equation}
M_{ij}
=
\int_\Omega \phi_i\phi_j \mathrm{d}x,\quad G_{ij} =
\int_\Omega \phi_i \mathcal G \phi_j \mathrm{d}x,\quad r_i(\beta) = \int_\Omega \frac{\delta E}{\delta u}(u_h)\phi_i \mathrm{d}x.
\end{equation}

Since $\mathbf M(\theta)$ is symmetric positive definite
whenever the neural basis functions are linearly independent,
the above system defines a finite-dimensional dynamical system.

\subsection{Energy Dissipation Preservation}

We show that the mixed neural Galerkin system
inherits the intrinsic energy dissipation property
of the continuous gradient flow.

\medskip
\begin{theorem}
    Assume that the operator $\mathcal G$ is symmetric
negative semi-definite and that the neural basis functions
are linearly independent so that the mass matrix
$\mathbf M(\theta)$ is symmetric positive definite.
Then the semi-discrete solution satisfies
\begin{equation}
\frac{\mathrm{d}}{\mathrm{d}t}E[u_h(t)]\le0 .
\end{equation}
\label{thm:energy_stable}
\end{theorem}

\begin{proof}
 Define the discrete energy
\begin{equation}
\mathcal E(t)=E[u_h(\cdot,t)] .
\end{equation}
By the chain rule and the definition of the variational derivative, we have
\begin{equation}
\frac{\mathrm{d}}{\mathrm{d}t}\mathcal E(t)
=\int_\Omega
\frac{\delta E}{\delta u}(u_h)
\partial_t u_h\mathrm{d}x .
\end{equation}
From the second equation of the mixed formulation,
we have
\begin{equation}
\int_\Omega v_h \psi\mathrm{d}x
=
\int_\Omega
\frac{\delta E}{\delta u}(u_h)\psi\mathrm{d}x,
\quad
\forall \psi\in\mathcal V_\theta .
\end{equation}
Since the neural basis functions are time-independent,
we have $\partial_t u_h \in \mathcal V_\theta$.
Choosing $\psi=\partial_t u_h$ yields
\begin{equation}
\frac{\mathrm{d}}{\mathrm{d}t}\mathcal E(t)
=
\int_\Omega v_h \partial_t u_h\mathrm{d}x .
\end{equation}
Using the first equation of Eq. \eqref{eq::mix_weakform} and
choosing $\psi=v_h$, we obtain
\begin{equation}
\int_\Omega v_h \partial_t u_h\mathrm{d}x
=
\int_\Omega v_h (\mathcal G v_h)\mathrm{d}x .
\end{equation}
Since $\mathcal G$ is symmetric negative semidefinite,
\begin{equation}
\int_\Omega v_h (\mathcal G v_h)\mathrm{d}x \le 0,
\end{equation}
which immediately implies
\begin{equation}
\frac{\mathrm{d}}{\mathrm{d}t}E[u_h(t)] = \int_\Omega v_h (\mathcal G v_h)\mathrm{d}x \le0 .
\end{equation}
\end{proof}

Theorem \ref{thm:energy_stable} establishes that the proposed DNN-Galerkin
formulation preserves the energy dissipation property
at the semi-discrete level. In practical computations,
the resulting ODE system must be further discretized
in time. To illustrate how the dissipation structure
can be preserved at the fully discrete level,
we present a representative example for the
$L^2$ gradient flow.

\begin{example}[Fully discrete energy-stable scheme for $L^2$ gradient flow]
    Consider the $L^2$ gradient flow, where the metric operator
reduces to the negative identity, $\mathcal{G} = -I$.
In this case the operator matrix becomes
$\mathbf{G}(\theta) = -\mathbf{M}(\theta)$,
and the semi-discrete system \eqref{eq::ode_system}
reduces to
\begin{equation}
    \mathbf{M}(\theta)\dot{\beta} = - \nabla_\beta \mathcal{E}(\beta).
\end{equation}

For a typical Ginzburg--Landau free energy functional
\begin{equation}
    E[u] = \int_\Omega
\left(
\frac{\epsilon^2}{2} |\nabla u|^2 + F(u)
\right)\mathrm{d}x,
\end{equation}
the gradient of the discrete energy can be evaluated via
integration by parts, yielding a decomposition:
\begin{equation}
    \mathbf{M}(\theta)\dot{\beta}
= - \epsilon^2 \mathbf{K}\beta - \mathbf{g}(\beta),
\end{equation}
where 
\begin{equation}
K_{ij} = \int_\Omega \nabla \phi_i \cdot \nabla \phi_j \mathrm{d}x,\quad g_i(\beta) = \int_\Omega F'(u_h) \phi_i \mathrm{d}x.
\end{equation}

To integrate this semi-discrete system while preserving
the dissipative structure, we adopt the energy-stabilized
two-stage IMEX scheme proposed in \cite{fu_energydiminishingimplicitexplicit_2024}.
Artificial stabilization parameters $\alpha \ge 0$
and $S \ge 0$ are introduced to guarantee unconditional
energy dissipation. The diffusion term is
treated implicitly, while the nonlinear term is
handled explicitly.

The fully discrete IMEX-RK2 scheme is given as follows:
\begin{equation}
\begin{aligned}
\textbf{Stage 1:}\qquad
\mathbf{M} \frac{\beta^* - \beta^n}{\gamma \Delta t} 
&= - \epsilon^2 \mathbf{K} \beta^* - \mathbf{g}(\beta^n) \\
&\quad - \alpha \epsilon^2 \mathbf{K} (\beta^* - \beta^n)
      - S \mathbf{M} (\beta^* - \beta^n), \\[1ex]
\textbf{Stage 2:}\quad
\mathbf{M} \frac{\beta^{n+1} - \beta^n}{\Delta t} 
&= - \epsilon^2 \mathbf{K} \big[(1-\gamma)\beta^* + \gamma \beta^{n+1}\big] \\
&\quad - \big[\delta \mathbf{g}(\beta^n) + (1-\delta)\mathbf{g}(\beta^*)\big] \\
&\quad - \alpha \epsilon^2 \mathbf{K}
      \big[\gamma(\beta^{n+1}-\beta^*) + \delta(\beta^*-\beta^n)\big] \\
&\quad - S \mathbf{M}
      \big[\gamma(\beta^{n+1}-\beta^*) + \delta(\beta^*-\beta^n)\big].
\end{aligned}
\label{eq::imex-rk2}
\end{equation}
Here $\gamma = 1 - 1/\sqrt{2}$ and
$\delta = 1 - 1/(2\gamma)$.
The resulting fully discrete DNN-Galerkin scheme
is second-order accurate in time and unconditionally
energy stable. The detailed proof of the discrete
energy law can be found in
\cite{fu_energydiminishingimplicitexplicit_2024}
and is therefore omitted here.
\end{example}

This example demonstrates that the semi-discrete
DNN-Galerkin formulation can be naturally combined
with existing structure-preserving time integrators
to construct fully discrete schemes that retain the
energy dissipation property.

\begin{remark}
    The semi-discrete system \eqref{eq::ode_system} can be written as
\begin{equation}
\frac{\mathrm{d}\beta}{\mathrm{d}t}
=
\mathbf{M}^{-1}\mathbf{G}\mathbf{M}^{-1} \mathbf{r}(\beta).
\end{equation}
Although the energy dissipation property is independent 
of the specific neural network architecture 
and holds for any linearly independent neural basis, 
 the conditioning of $\mathbf{M}$ 
strongly affects numerical stability in practice. 
Nearly linearly dependent basis functions may lead to 
an ill-conditioned system, motivating the adaptive 
basis construction introduced in the next section.
\end{remark}

%SECTION 3
\section{Choice of Neural Basis Functions}\label{sec::basis_choice}
While Theorem 2.1 guarantees unconditional semi-discrete energy stability,
the overall accuracy and robustness of the method are primarily determined
by the quality of the neural trial space.
Unlike traditional discretizations where basis functions are fixed
(e.g., Fourier modes or finite elements),
the DNN Galerkin framework allows the trial space
\(\mathcal{V}_\theta\) to be generated and adapted through neural features.
Consequently, the construction of neural basis functions
becomes a central component of the proposed methodology.
This section presents a strategy for constructing
and refining the neural basis, including two approaches illustrated in Figure \ref{fig:basis_evolution_parallel}:

\begin{itemize}
    \item \textbf{Randomized Basis Construction.}
    A static neural basis is first generated under the Random Features paradigm,
    providing universal approximation without iterative training.
    On this foundation, a structured first-layer initialization (SFLI)
    is introduced to enhance linear independence and improve
    the spectral properties of the mass matrix.

    \item \textbf{Problem-Informed Basis Refinement.}
    Starting from the structured static basis $\Phi(x,t;\theta_0)$,
    equation-specific information is incorporated either
    through initial-condition alignment or through residual-driven
    pre-training.
    This produces either a frozen problem-adapted basis
    or a time-local adaptive basis.
\end{itemize}

\usetikzlibrary{shapes.geometric, arrows.meta, positioning, calc, backgrounds, fit}

\definecolor{basisblue}{RGB}{0, 102, 204}
\definecolor{optorange}{RGB}{230, 126, 34}
\definecolor{lossgray}{RGB}{245, 245, 245}

\begin{figure}[htbp]
    \centering
    \begin{tikzpicture}[
        node distance=1.2cm and 1.8cm,
        box/.style={rectangle, draw=basisblue, fill=white, thick, minimum width=2.4cm, minimum height=1.0cm, align=center, font=\small\sffamily},
        loss_node/.style={ellipse, draw=optorange, fill=lossgray, thick, minimum width=2.2cm, align=center, font=\scriptsize\itshape\sffamily},
        final_node/.style={rectangle, rounded corners=3pt, draw=black!70, fill=gray!5, thick, minimum width=3.2cm, minimum height=1.0cm, align=center, font=\small\sffamily},
        arrow/.style={-{Stealth[scale=1.1]}, thick, draw=black!80},
        label_above/.style={font=\footnotesize\itshape\sffamily, midway, yshift=0.25cm},
        label_static/.style={font=\footnotesize\bfseries\sffamily, color=basisblue!80, yshift=0.15cm}
    ]

    \node [font=\large\bfseries\sffamily] (x) {$\begin{pmatrix} x \\ t \end{pmatrix}$}; 
    
    \node [box, right=of x] (hidden) {Hidden Layers \\ ($\theta$)};
    \draw [arrow] (x) -- node[label_above] {(SFLI)} (hidden);

    \node [box, right=of hidden, line width=1.2pt] (phi0) {$\Phi(x, t;\theta_0)$};
    \draw [arrow] (hidden) -- node[label_above] {Initialization} (phi0);
    
    \node [label_static, above=0.05cm of phi0] {Static Neural Basis};

    %
    %   \node [box, below=2.8cm of x] (phi0_in) {$\Phi(x, t;\theta_0)$};
    \node [box, below=2.8cm of x, xshift=-0.8cm] (phi0_in) {$\Phi(x, t;\theta_0)$};

    \node [loss_node] (loss_init) at (phi0_in -| hidden) [yshift=0.8cm] {Supervised Learning \\ \tiny $\| \beta^T \Phi - u_0 \|^2$};
   % \node [loss_node] (loss_res) at (phi0_in -| hidden) [yshift=-0.8cm] {$\mathcal{L}_{\text{res}}$ \\ \tiny $\| \partial_t u - \mathcal{L} u \|^2$};
\node [loss_node] (loss_res) at (phi0_in -| hidden) [yshift=-0.8cm] {PINN};

    \node [final_node] (basis_frozen) at (loss_init -| phi0) {$\Phi(x; \theta^*)$ \\ \footnotesize (Frozen Basis)};
    \node [final_node] (basis_adaptive) at (loss_res -| phi0) {$\Phi(x, t_k; \theta^*)$ \\ \footnotesize (Adaptive Basis)};

    \draw [arrow] (phi0_in.east) -- ++(0.5,0) |- (loss_init.west);
    \draw [arrow] (phi0_in.east) -- ++(0.5,0) |- (loss_res.west);
    \draw [arrow] (loss_init) -- (basis_frozen);
    \draw [arrow] (loss_res) -- (basis_adaptive);

    \begin{scope}[on background layer]
        \node[draw=optorange!60, fill=optorange!5, dashed, thick, rounded corners=10pt, inner sep=15pt, 
              fit=(phi0_in) (loss_init) (loss_res) (basis_frozen) (basis_adaptive)] (training_box) {};
    \end{scope}
    
    \node[anchor=south, font=\small\bfseries\sffamily, color=optorange, yshift=0.2cm] at (training_box.north) {Pre-Training / Optimization};

    \end{tikzpicture}
    \caption{Construction of the neural basis functions. 
%    Top: Generation of the static random basis $\Phi(x, t;\theta_0)$ with and without SFLI. 
%    Bottom: Optimization stage where $\Phi(x, t;\theta_0)$ is refined via either initial condition alignment ($\mathcal{L}_{\text{init}}$) or PDE residual guidance ($\mathcal{L}_{\text{res}}$).
    }
    \label{fig:basis_evolution_parallel}
\end{figure}
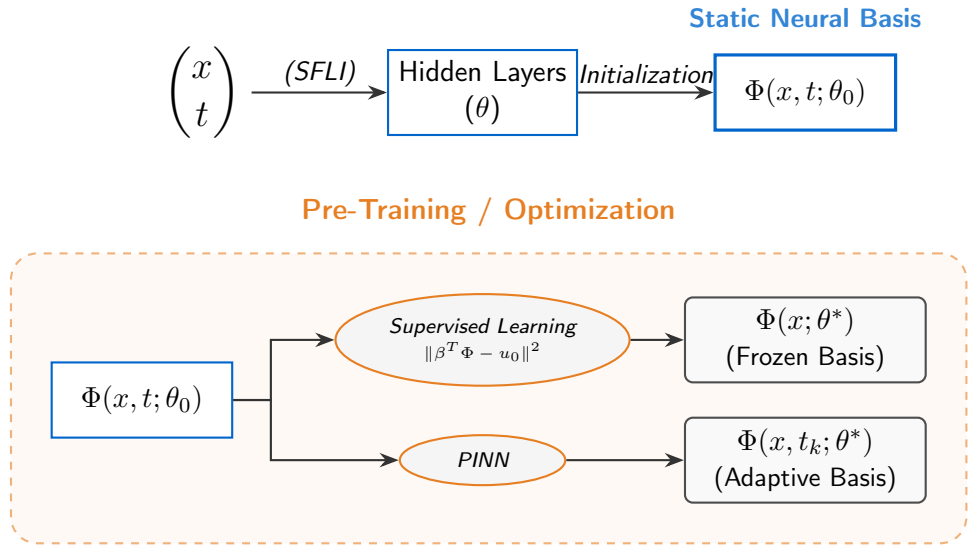

\subsection{Static Neural Basis Functions}
In certain regimes, the neural basis can be constructed
without iterative training.
By carefully initializing the parameters $\theta$,
one may generate a trial space with desirable approximation
and conditioning properties,
thereby avoiding the costly training process.
Throughout this subsection, the parameters $\theta$ are fixed in time,
and only the coefficient vector $\beta(t)$ evolves.

\subsubsection{Random Features}
The Random Feature paradigm provides a direct mechanism
for constructing a neural basis without backpropagation.
This idea originates from the extreme learning machine \cite{huang_universalapproximationusing_2006} and 
random Fourier features \cite{rahimi_randomfeatureslargescale_2007a}, 
which establish that randomly sampled hidden nodes
yield universal approximation with high probability.
This theoretical foundation has been extensively verified across various computational domains 
\cite{chen_randomfeaturemethod_2023, nelsen_randomfeaturemodel_2021, dong_localextremelearning_2021, chen_optimizationrandomfeature_2024}. 

In the Galerkin framework,
the internal parameters $\theta=\{W_l,b_l\}_{l=1}^L$
are initialized from a prescribed distribution
(e.g., Glorot or Xavier initialization)
and subsequently frozen.
The resulting trial space is
\begin{equation}
\mathcal{V}_{\theta_\mathrm{rand}}
=
\mathrm{span}
\{\phi_1(x;\theta_{\mathrm{rand}}),
\dots,
\phi_m(x;\theta_{\mathrm{rand}})\}.
\end{equation}

Under this construction,
the nonlinear PDE reduces to a system of ODEs \eqref{eq::ode_system} that are linear
with respect to the coefficient vector $\beta(t)$.

Although Random Features offer simplicity
and avoid training plateaus,
their approximation efficiency deteriorates
as higher accuracy is required.
In practice, the number of neurons must grow rapidly
to enrich the functional coverage,
which often leads to highly correlated basis functions.
Consequently, the mass matrix $\mathbf{M}$
may develop small singular values,
resulting in poor conditioning.
This limitation motivates the structured initialization
discussed next.
\subsubsection{Structured First-Layer Initialization (SFLI)}
To move beyond purely stochastic initialization,
we adopt a deterministic structured construction
for the first network layer \cite{tang_structuredfirstlayerinitialization_2025}. For a one-dimensional domain $\Omega$ with identity input mapping $z_0=x$,
the first-layer neurons are initialized as
\begin{equation}
(z_1)_j = \sigma(s( x-x_j)),
\quad j=1,\dots,m,
\end{equation}
where $\{x_j\}_{j=1}^m$ are uniformly distributed points
and $s$ controls the slope of the activation. 
 
This structured placement ensures that
the initial basis functions are spatially localized
and well separated.
As a result, the $\epsilon$-rank \cite{yang_$e$rankstaircasephenomenon_2025} of the mass matrix
is significantly improved compared to purely random sampling.
The enhanced spectral structure directly benefits
the conditioning of the semi-discrete system.
 
From a numerical analysis perspective,
SFLI can be interpreted as a mesh-free analogue
of localized finite element basis functions.
The structured placement of neurons
provides spatial localization similar to nodal basis functions,
while deeper network layers allow nonlinear deformation
of these initial features.

\subsection{Offline Pre-training Basis}
While the SFLI provides a robust general neural basis, 
the trial space $\mathcal{V}_\theta$ can be further enriched
 by incorporating problem-specific information. 
We discuss two refinement strategies:
initial-condition alignment and residual-driven pre-training.

\subsubsection{Initial-Conditions-Motivated Basis}
A natural strategy to customize the functional basis is to align the neural 
network with the initial state $u_0(x)$ through an offline pre-training stage. 
This is formulated as the following $L^2$ minimization problem:
\begin{equation}    
    \min_{\theta, \beta} \mathcal{L}_{\text{init}}(\theta, \beta) = \frac{1}{N} \sum_{i=1}^N \left| \beta^T \Phi(x_i; \theta) - u_0(x_i) \right|^2.
\end{equation}
Here, $\beta$ serves as an auxiliary parameter
during optimization and is discarded afterward.
Upon convergence,
the internal parameters are frozen as $\theta^*$,
and the trial space becomes
\begin{equation}
\mathcal{V}_{\theta^*}
=
\text{span}\{\phi_1(x;\theta^*),\dots,\phi_m(x;\theta^*)\}.
\end{equation}
This procedure produces a problem-adapted frozen basis
that accurately captures the initial geometry.

However, two limitations arise.
First, if $u_0$ lies on a low-dimensional manifold,
the optimization may reduce basis diversity,
leading to a low effective rank of $\mathbf{M}$.
Second, for problems exhibiting significant
topological transitions,
a basis optimized solely at $t=0$
may lack sufficient expressivity at later times.

\subsubsection{Adaptive Basis Strategy}
For problems involving significant topological changes or complex dynamics,
a frozen basis may become insufficient for long-term simulations.
We therefore introduce an adaptive strategy
that combines global feature extraction
with local Galerkin evolution.
The adaptive DNN-Galerkin procedure consists of three
main components: 
\begin{enumerate}
    \item \textbf{Global Feature Extraction via Pre-training.} 
    Over the entire time interval $[0, T]$, 
    a PINN is first trained to capture the coarse global dynamics
by minimizing the residual functional
\begin{equation}
\mathcal{L}_{\mathrm{res}}
=
\int_0^T \int_\Omega
\left|
\partial_t u_p - \mathcal{G}\frac{\delta E}{\delta u}(u_p) 
\right|^2
\mathrm{d}x\mathrm{d}t,
\end{equation}
where $u_p(x,t)=\beta^T\Phi(x,t;\theta)$.
The optimized hidden-layer parameters $\theta^*$
are retained as a global feature extractor.
High accuracy is not required at this stage,
since the network serves only as a feature extractor
for constructing the neural basis.

    \item \textbf{Time-Local Basis Construction.} 
    Define a set of restart times $\mathcal{T}_{\text{list}} = \{T_0, T_1, \dots, T_K\}$, where 
    $0 = T_0 < T_1 < \dots < T_K = T$. 
    At each restart time $T_k$, a local neural basis is constructed as
\begin{equation}
\Phi_k(x)
=
\big[
u_h(x,T_k^-),
\phi_1(x,T_k;\theta^*),
\dots,
\phi_m(x,T_k;\theta^*),
1
\big]^T.
\end{equation}
The solution on the subinterval $[T_k,T_{k+1}]$
is represented as
\begin{equation}
    u_{h}(x,t)=\beta(t)^T\Phi_k(x),\quad t\in[T_k,T_{k+1}).
\end{equation}
For $k=0$, the representation is initialized by setting
$u_h(x,T_0^-)=u_0(x)$ so that the initial condition is
satisfied.
    \item \textbf{Galerkin Evolution.} 
Within each subinterval $[T_k,T_{k+1}]$, the system is
advanced by solving the projected ODE system in the
fixed trial space $\mathcal{V}_k$,
\begin{equation}
    \mathbf{M}_k \frac{\mathrm{d}\beta}{\mathrm{d}t} = \mathbf{G}_k \mathbf{M}_k^{-1} \mathbf{r}_k(\beta),
\end{equation}
with the initial coefficient vector $\beta(T_k) = [1, 0, \dots, 0]^T$.
\end{enumerate}

The overall adaptive DNN-Galerkin procedure is summarized
in Algorithm~\ref{alg:hybrid_galerkin}. An important property of the above adaptive procedure is
that the solution is preserved exactly when the neural
basis is updated. This feature ensures that the discrete
free energy does not experience artificial jumps during
basis transitions. As a consequence, the global energy
dissipation property of the fully discrete scheme can be
established.

\begin{algorithm}[H]
\caption{Adaptive DNN Galerkin Scheme with PINN Pre-training}
\label{alg:hybrid_galerkin}
\begin{algorithmic}[1]
\REQUIRE Initial condition $u_0(x)$, time interval $[0, T]$, subinterval size $\Delta T$, and network architecture.
\ENSURE Numerical solution $u(x, t)$ for $t \in [0, T]$.

\STATE \textbf{Stage I: PINN Pre-training}
\STATE Train a standard PINN $u_p(x, t; \theta^*)$ on $[0, T]$ by minimizing the residual loss $\mathcal{L}_{\text{res}}$.
\STATE Extract the learned parameters $\theta^*$ for the hidden layers.

\STATE \textbf{Stage II: Temporal Galerkin Evolution}
\STATE Initialize $u_{0}(x, T_0) = u_0(x)$.
\FOR{$k = 0, 1, \dots, K - 1$}
    \STATE \textbf{Basis Construction}: Form the augmented basis at $T_k$:
    $\Phi_k(x) = [u_{h}(x, T_k), \phi_1(x, T_k; \theta^*), \dots, \phi_m(x, T_k; \theta^*), 1]^T$.
    \STATE \textbf{ODE Assembly}: Construct the mass matrix $\mathbf{M}_k$ and residual vector $\mathbf{r}_k$.
    \STATE \textbf{Time Stepping}: Solve the ODE system over $[T_k, T_{k+1}]$ using structure-preserving time integrator (e.g., IMEX-RK2):
    
    $\mathbf{M}_k \frac{\mathrm{d}\beta}{\mathrm{d}t} = \mathbf{G}_k \mathbf{M}_k^{-1} \mathbf{r}_k(\beta), \quad \beta(T_k) = [1, 0, \dots, 0]^T$.
    \STATE \textbf{Update}: Set $u_{h}(x, t) = \beta(t)^T \Phi_k(x)$ for $t \in [T_k, T_{k+1}]$.
\ENDFOR
\end{algorithmic}
\end{algorithm}

\begin{theorem}[Global discrete energy dissipation of the adaptive DNN-G scheme]

Let $\{t_n\}_{n=0}^{N_t}$ be the fully discrete time sequence and
$\{T_k\}_{k=0}^K\subseteq\{t_n\}_{n=0}^{N_t}$ the set of basis update times.
Assume that on each subinterval $[T_k,T_{k+1}]$ the
semi-discrete DNN-Galerkin system is advanced by a fully
discrete time integrator that satisfies the discrete
energy dissipation property (for example, the IMEX-RK2
scheme \eqref{eq::imex-rk2}).

Then the adaptive DNN-Galerkin solution
$u_h^{n} := u_h(\cdot,t_n)$ satisfies the global discrete
energy dissipation law
\begin{equation}
E[u_{h}^{n+1}] \le E[u_h^{n}], \qquad \forall n\ge0 .
\end{equation}

That is, the discrete free energy is monotonically
non-increasing over the entire simulation, including
across the basis update points.
\label{thm::fully_energy_stable}
\end{theorem}

\begin{proof}

On each subinterval $[T_k,T_{k+1}]$, the neural trial space
$\mathcal{V}_k$ remains fixed. The Galerkin discretization
then yields a finite-dimensional dynamical system for the
coefficient vector $\beta(t)$. The resulting evolution equation
can be written in the gradient flow form
\begin{equation}
\dot{\beta} = \mathbf{Q}\nabla_\beta \mathcal E(\beta),
\end{equation}
where $\mathcal E(\beta)=E[u_h]$ denotes the discrete
energy and $\mathbf{Q} = \mathbf{M}^{-1}\mathbf{G}\mathbf{M}^{-1}$ is a symmetric negative semidefinite matrix
determined by the mass matrix of the basis functions.
Therefore, the semi-discrete system is itself a
finite-dimensional gradient flow.

By assumption, the time integration scheme applied to this
system satisfies a discrete energy dissipation property.
Consequently, the numerical solution satisfies
\begin{equation}
E[u_{h}^{n+1}] \le E[u_h^{n}],
\end{equation}
for all time steps $t_n,t_{n+1}\in[T_k,T_{k+1}]$.

Next consider a basis update time $T_k$.
The adaptive construction $\Phi_k$ includes the previous
numerical solution $u_h(\cdot,T_k^-)$ as the first basis
function of the new trial space.
With the initial coefficient vector
\begin{equation}
\beta(T_k^+)=[1,0,\dots,0]^T,
\end{equation}
the solution is preserved exactly,
\begin{equation}
u_h(x,T_k^+)
=
\beta(T_k^+) \cdot \Phi_k(x)
=
u_h(x,T_k^-).
\end{equation}
Consequently, the discrete energy is continuous at the
update time,
\begin{equation}
E[u_h(\cdot,T_k^+)]=E[u_h(\cdot,T_k^-)].
\end{equation}

Combining the energy dissipation within each subinterval
with the continuity of the energy at the update points
yields the global estimate
\begin{equation}
E[u_{h}^{n+1}] \le E[u_h^{n}], \qquad \forall n\ge0 .
\end{equation}
\end{proof}

%%SECTION 4
\section{Numerical Experiments} \label{sec::numerical}
\label{sec:experiments}
\subsection{Settings}
Unless otherwise specified in the individual examples, all neural networks 
considered in this work are fully-connected feedforward networks with hyperbolic tangent ($\tanh$) 
activation functions. The resulting semi-discrete Galerkin system 
\eqref{eq::ode_system} is solved using a second-order stabilized Implicit-Explicit Runge-Kutta method (IMEX-RK2) \cite{fu_energydiminishingimplicitexplicit_2024} 
with a uniform time step size $\Delta t = 10^{-3}$. 
Spatial integrals in the Galerkin formulation are evaluated using Gaussian 
quadrature rules or Monte Carlo integration for mass matrix $\mathbf{M}$ and stiffness matrix $\mathbf{K}$, where 
\begin{equation}
\mathbf{M}_{ij}
=
\int_{\Omega} \phi_i\phi_j\mathrm{d}x,\quad \mathbf{K}_{ij} = \int_{\Omega} \nabla \phi_i \cdot \nabla \phi_j\mathrm{d}x. 
\end{equation}

Neural basis functions are constructed from networks whose weights are initialized using the LeCun initialization,
and biases are initialized to zero, resulting in random neural features.
When specified, the structured first-layer initialization (SFLI) is 
applied to the first hidden layer. 

When pre-training is employed, it corresponds to a standard Physics-Informed 
Neural Network (PINN) training procedure, which also serves as a baseline method 
in the numerical comparisons. The PINN training is performed using the Adam 
optimizer for 30000 steps, with an initial learning rate of $10^{-3}$ and a cosine-decay 
learning-rate schedule. Training samples are uniformly and randomly generated in 
the spatio-temporal domain and regenerated at every optimization step.

To mitigate the ill-conditioning of the mass matrix caused by nearly linearly 
dependent neural basis functions, we apply a discrete orthogonalization 
procedure based on truncated singular value decomposition. 
At each restart time $T_k$, the raw neural basis functions are evaluated at Gaussian quadrature nodes and assembled into a weighted basis matrix.
Singular values below a prescribed tolerance relative to the largest singular value 
are truncated to construct an approximately $L^2$-orthonormal basis. The initial 
condition basis function $u_h(x,T_k^-)$ is explicitly preserved in the resulting augmented
basis. This preprocessing step improves the conditioning of the semi-discrete system while preserving the approximation capacity of the trial space.

The numerical accuracy is evaluated using the relative $L^2$ error

\begin{equation}
e(t) =
\frac{\|u_h(\cdot,t)-u_{\mathrm{ref}}(\cdot,t)\|_{L^2(\Omega)}}
{\|u_{\mathrm{ref}}(\cdot,t)\|_{L^2(\Omega)}},
\end{equation}
where $u_{\mathrm{ref}}$ denotes a highly resolved reference solution or an analytical solution. 
For the Allen--Cahn equations, the reference solutions are computed using spectral 
methods in space combined with the fourth-order exponential time-differencing 
Runge--Kutta scheme \cite{kassam_fourthordertimesteppingstiff_2005}.

The network architectures and example-specific training hyperparameters are summarized 
in Table~\ref{tab:hyperparameters}.

\begin{table}[htbp]
	\centering
	\caption{Summary of network architectures and example-specific hyperparameters for each numerical example.}
	\label{tab:hyperparameters}
	% Scale the table to fit the text width
	\resizebox{\textwidth}{!}{
		\begin{tabular}{lccccc}
			\toprule
			\textbf{Example} & \textbf{Hidden Layers} & \textbf{Batch Size} & \textbf{Opt. Steps} & \textbf{Quad. Pts} & \textbf{Subintervals ($\mathcal{T}_{\text{list}}$)} \\
			\midrule
			\ref{eg:heat_high_d} (5D Heat) & $[50, 50, 50, 900]$ & $10,000$ & $30,000^*$ & $12$/dim & $[0.0, 1.0]$ \\
			\ref{eg:heat_high_d} (10D Heat) & $[50, 50, 50, 900]$ & $10,000$ & $30,000^*$ & $N_{\text{MC}}$=$10^4$ to $10^7$ & $[0.0, 1.0]$ \\
			\ref{eg:1d_AC} (1D AC) & $[128, 128, 128]$ & $8,192$ & $30,000$ & $1024$ & $\{0.0, 0.2, \dots, 1.0\}$ \\
			\ref{eg:ac_2d_circle} (2D AC, Bubble) & $[50, 50, 50, 64/1024]$ & $15,000$ & $30,000^*$ & $128$/dim & $\{0.0, 0.1, \dots, 1.0\}$ \\
			\ref{eg:ac_2d_star} (2D AC, Star) & $[50, 50, 50, 1024]$ & $20,000$ & $30,000^*$ & $128$/dim & $\{0.0, 0.5, \dots, 5.0\}^\dagger$ \\
			\ref{eg:ac_2d_random} (2D AC, Random) & $[50, 50, 50, 1024]$ & $10,000$ & $30,000^*$ & $128$/dim & $\{0.0, 0.1, \dots, 1.0\}^\dagger$ \\
			\ref{eg:ch_1d} (1D CH) & $[128, 128, 128]$ & $8,192$ & $30,000$ & $1024$ & $\{0.0, 0.005, 0.01, 0.015, 0.04, 0.08, 0.12, 0.16, 0.2\}$ \\
			\bottomrule
		\end{tabular}
	} % End of resizebox
	
	\vspace{0.15cm}
	\parbox{\textwidth}{\footnotesize
		\textit{Note:} All networks use $\tanh$ activations. 
        The optimization steps correspond to the number of Adam iterations.
		$^*$Indicates an additional 1500 L-BFGS refinement iterations with the batch size increased by a factor of four.
        The subintervals ($T_{\text{list}}$) indicate the restart points for the adaptive basis in time evolution.
        $^\dagger$For the untrained SFLI-DNN-G scheme in Ex.~\ref{eg:ac_2d_star} and \ref{eg:ac_2d_random}, 
        a single global interval ($[0.0,5.0]$ and $[0.0,1.0]$, respectively) is employed instead of adaptive subintervals. }
\end{table}

\subsection{Heat Equation}
To examine the high-dimensional behavior and temporal accuracy 
of the proposed DNN-Galerkin method, 
we consider the $d$-dimensional heat equation, 
which serves as a prototypical linear $L^2$ gradient flow. 
As the spatial dimension increases, classical mesh-based discretizations 
suffer from the exponential growth of grid points, making them 
computationally prohibitive. 
This example therefore provides a natural benchmark 
for assessing the approximation efficiency of neural trial spaces 
in high-dimensional settings.
\begin{example}[High-dimensional Heat Equation]
    Consider the $d$-dimensional heat equation on the domain $\Omega = [-0.5, 0.5]^d$ with homogeneous
     Dirichlet boundary conditions:
    \begin{equation}
        u_t = \frac{1}{d\pi^2}\Delta u, \quad \mathbf{x} \in \Omega, \quad t \in [0, 1],
    \end{equation}
   which corresponds to the $L^2$ gradient flow associated with the Dirichlet energy
   \begin{equation}
    E[u] = \frac{1}{2} \int_{\Omega} |\nabla u|^2 d\mathbf{x}.
\end{equation}
    The initial condition is chosen as:
    \begin{equation}
        u_0(\mathbf{x}) = \prod_{i=1}^d \cos(\pi x_i),
        \label{eq:heat_init}
    \end{equation}
    for which the exact solution reads:
    \begin{equation}
        u(\mathbf{x}, t) = e^{-t}\prod_{i=1}^d \cos(\pi x_i).
        \label{eq:heat_exact}
    \end{equation}
    \label{eg:heat_high_d}
\end{example}

The semi-discrete system \eqref{eq::ode_system} 
is integrated using second- and third-order 
Diagonally Implicit Runge-Kutta (DIRK) schemes \cite{kennedy_diagonallyimplicitrunge_2019}. 
The neural trial space consists of $m=900$ frozen random features 
generated from a fully connected network. 
Homogeneous Dirichlet boundary conditions are enforced by the modified basis
\[
\tilde{\phi}_i(\mathbf{x};\theta)
= \phi_i(\mathbf{x};\theta)\prod_{j=1}^d \left(\frac{1}{4}-x_j^2\right).
\]

All spatial integrals are evaluated using Gaussian quadrature for $d=5$ 
and Monte Carlo integration for higher dimensions.

Figure~\ref{fig:heat_5d_convergence} shows the temporal convergence results for $d=5$. 
The DNN-G scheme follows the expected convergence orders 
of the underlying DIRK integrators. 
This confirms that the Galerkin projection preserves the temporal accuracy 
of the time discretization within the neural trial space.

\begin{figure}[htbp]
    \centering
    \includegraphics[width=\textwidth]{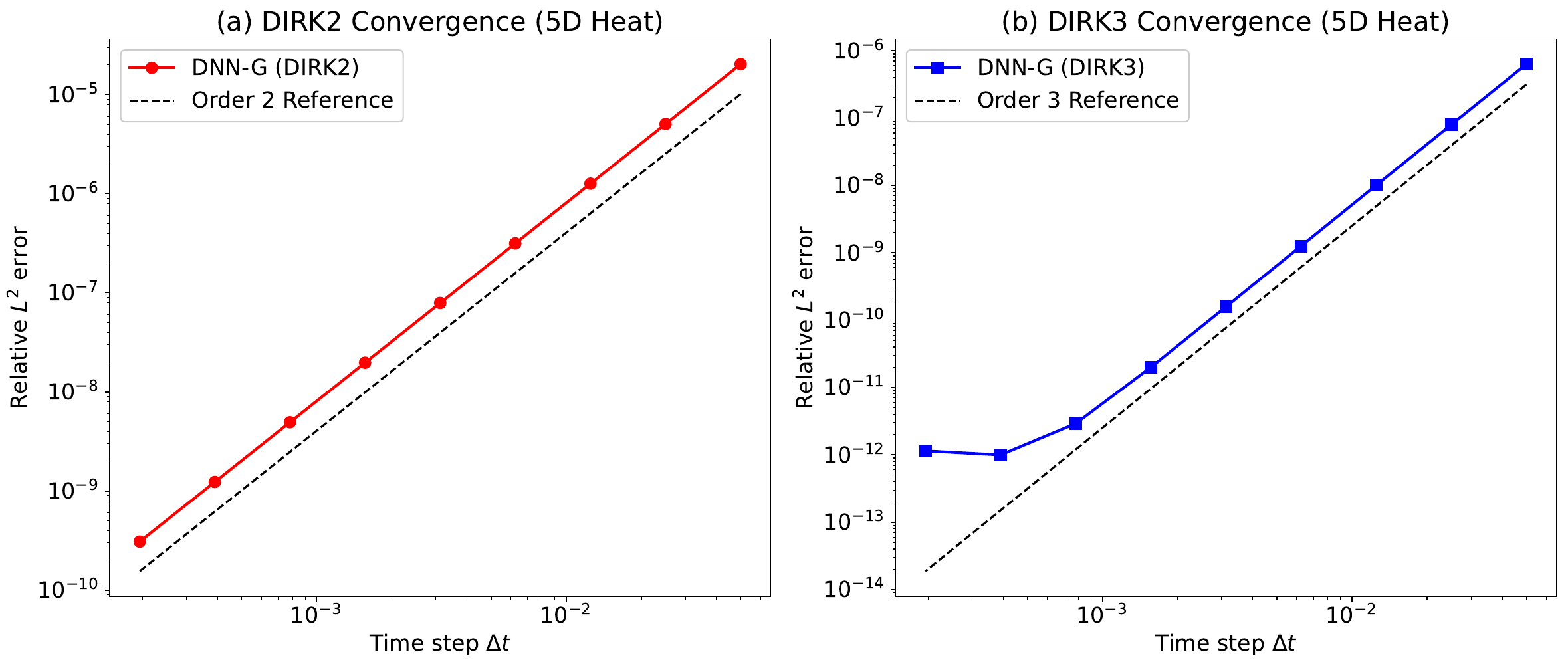}
    \caption{Temporal convergence test for the 5D heat equation (Example \ref{eg:heat_high_d}, $d=5$). 
    (a) The second-order DIRK scheme (DIRK2) exhibits an $O(\Delta t^2)$ convergence rate. 
    (b) The third-order DIRK scheme (DIRK3) achieves $O(\Delta t^3)$ accuracy, eventually reaching a plateau due to spatial discretization limits.}
    \label{fig:heat_5d_convergence}
\end{figure}

For the DIRK3 scheme, the error curve eventually reaches a plateau 
as $\Delta t$ decreases, indicating a transition from a time-discretization-dominated regime 
to a spatial-approximation-dominated regime. 
The plateau level reflects the intrinsic approximation error 
of the neural trial space.
To further assess the accuracy, we compare the relative $L^2$ error 
at $t=1$ with a standard PINN under the same setup. 
The PINN attains a relative $L^2$ error of $1.54 \times 10^{-4}$, 
whereas the DNN-G method achieves $2.57 \times 10^{-11}$ when $d=5$. 
This result highlights the accuracy advantage of the Galerkin formulation.

We further extend the test to a higher-dimensional regime ($d=10$), 
where the spatial integrals in the variational formulation are approximated 
using Monte Carlo integration. 
In this setting, the statistical accuracy of the sampled mass matrix 
becomes a dominant source of spatial error. 

As shown in Fig.~\ref{fig:heat_10d_mc_convergence}, 
once the number of quadrature points reaches $N_{\text{MC}}=10^6$, 
the DNN-G scheme consistently outperforms the standard PINN. 
Despite the statistical noise introduced by high-dimensional integration, 
the DNN-G framework preserves the gradient-flow structure. 
Its energy dissipation closely follows the analytical trajectory, 
whereas the optimization-based PINN exhibits noticeable deviations.

\begin{figure}[htbp]
    \centering
    \includegraphics[width=\textwidth]{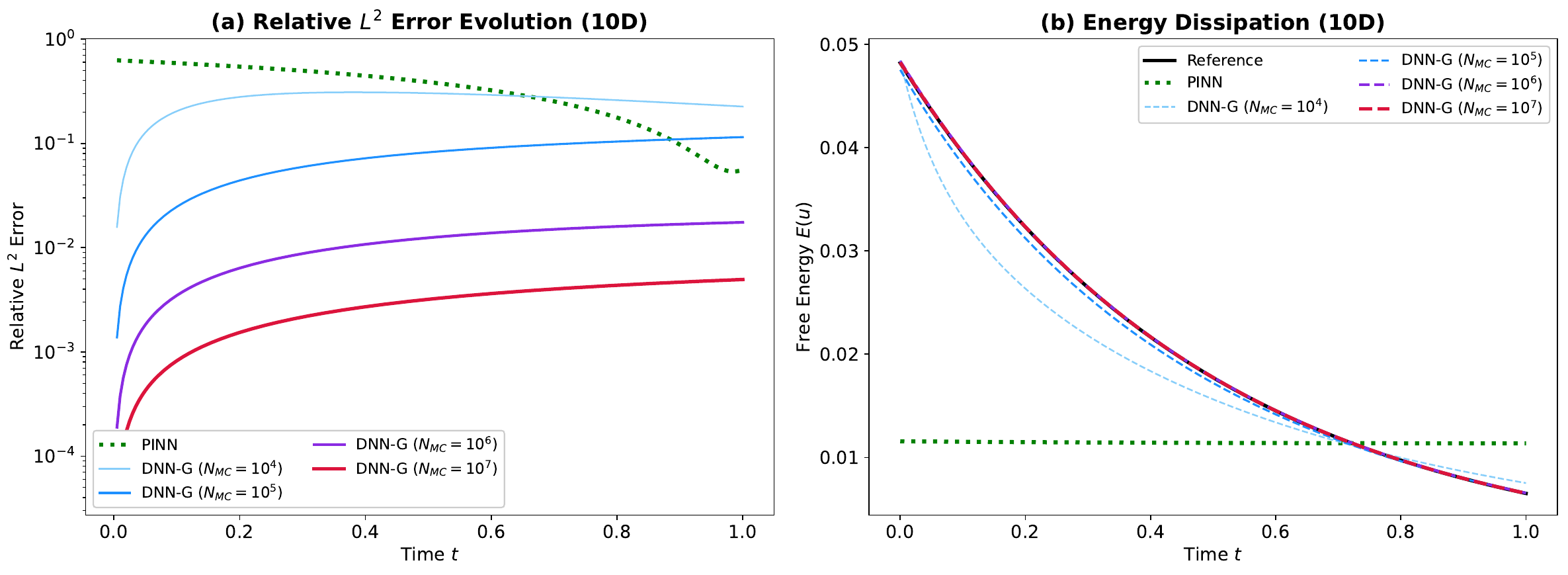}
    \caption{Quantitative comparison for the 10D heat equation (Example \ref{eg:heat_high_d}, $d=10$). (a) Relative $L^2$ error evolution. (b) Energy dissipation. 
    The standard PINN (green dotted line) is compared against the DNN-G scheme, where the integrals are approximated using varying numbers of Monte Carlo points
    ($N_{\text{MC}} = 10^4$ to $10^7$).}
    \label{fig:heat_10d_mc_convergence}
\end{figure}

We remark that this level of accuracy is obtained 
with only $m=900$ basis functions. 
For comparison, a finite difference discretization 
with approximately the same number of degrees of freedom 
($4^5 \approx 1024$ in five dimensions and $2^{10} \approx 1024$ in ten dimensions) 
would correspond to a mere four and two grid points per spatial direction, respectively.
This illustrates the efficiency of neural trial spaces 
for high-dimensional problems.

\subsection{Allen--Cahn Equation}

The Allen--Cahn equation is a classical phase-field model 
describing the $L^2$ gradient flow of the Ginzburg--Landau free energy 
\cite{allen_microscopictheoryantiphase_1979,shen_numericalapproximationsallencahn_2010}. 
It serves as a prototypical nonlinear dissipative system 
for studying phase separation and interface dynamics.
Given the Ginzburg--Landau free energy:
\begin{equation}
    E[u] = \int_{\Omega} \left( \frac{\epsilon^2}{2} |\nabla u|^2 + F(u) \right) \mathrm{d}x,
    \label{eq:free_energy_ac}
\end{equation}
where $F(u) = \frac{1}{4}(u^2 - 1)^2$ is the double-well potential, 
the corresponding $L^2$ gradient flow reads:
\begin{equation}
    u_t = \epsilon^2 \Delta u - \kappa f(u), \quad f(u) = F'(u) = u^3 - u.
\end{equation}
This equation models the phase separation process, where the system evolves towards minimizing the free energy by forming
interfaces between the two stable phases $u = \pm 1$.

%\subsubsection{1D Case}
\subsubsection{1D Case}

\begin{example}[1D Allen--Cahn equation]
We consider the one-dimensional problem:
    \begin{equation}
        u_t = \epsilon^2 u_{xx} - \kappa f(u), \quad x \in [-1, 1], \quad t \in [0, 1],
    \end{equation}
    where $\epsilon = 0.01$ and $\kappa = 5$. The initial condition is set as $u_0(x) = x^2\cos(\pi x)$.
    \label{eg:1d_AC}
\end{example}

In this example, the initial condition is enforced via a hard constraint:
$u_{p}(x,t;\theta) = u_0(x) + t \cdot \mathcal{N}(x,t;\theta)$.
The neural basis functions are pre-trained via this PINN-based procedure, 
after which the adaptive DNN-Galerkin (ADNN-G) framework 
is employed for time evolution. 
To discretize the semi-discrete system \eqref{eq::ode_system}, 
we consider two representative schemes.
\begin{itemize}
\item \textbf{Stabilized Semi-Implicit (SSI1) Scheme}. 
A first-order scheme treating the diffusion term implicitly 
and the nonlinear term explicitly. 
A linear stabilization term with parameter $S>0$ is introduced 
to enhance energy stability:
\begin{equation}
\mathbf{M} 
\frac{\beta^{n+1} - \beta^n}{\Delta t}
= -\epsilon^2 \mathbf{K}\beta^{n+1}
- \mathbf{g}(\beta^n)
- S\mathbf{M}(\beta^{n+1} - \beta^n),
\end{equation}
where \begin{equation}
    g_j(\beta^n)
    = \int_{-1}^1 f(u_h(x,t_n)) \phi_j(x) \mathrm{d}x.
\end{equation}
 
\item \textbf{Energy-Stabilized IMEX-RK2 Scheme} \eqref{eq::imex-rk2}.
% For the semi-discrete system, we adopt the energy-stabilized two-stage IMEX strategy proposed in \cite{fu_energydiminishingimplicitexplicit_2024}. 
% To unconditionally preserve the energy dissipation property, stabilization parameters $\alpha \ge 0$ and $S \ge 0$ are introduced. 
% The stiff linear diffusion term is treated implicitly while the nonlinear force term is handled explicitly. 
% The fully discrete scheme computes the intermediate state $\beta^*$ and the final state $\beta^{n+1}$ as follows:

% \textbf{Stage 1:}
% \begin{equation}
%     \begin{aligned}
%         \mathbf{M} \frac{\beta^* - \beta^n}{\gamma \Delta t} 
%         &= - \epsilon^2 \mathbf{K} \beta^* - \mathbf{r}(\beta^n) \\
%         &\quad - \alpha \epsilon^2 \mathbf{K} (\beta^* - \beta^n) 
%         - S \mathbf{M} (\beta^* - \beta^n),
%     \end{aligned}
% \end{equation}

% \textbf{Stage 2:}
% \begin{equation}
%     \begin{aligned}
%         \mathbf{M} \frac{\beta^{n+1} - \beta^n}{\Delta t} 
%         &= - \epsilon^2 \mathbf{K} \big[ (1-\gamma) \beta^* + \gamma \beta^{n+1} \big] 
%         - \big[ \delta \mathbf{r}(\beta^n) + (1-\delta) \mathbf{r}(\beta^*) \big] \\
%         &\quad - \alpha \epsilon^2 \mathbf{K} \big[ \gamma(\beta^{n+1} - \beta^*) + \delta(\beta^* - \beta^n) \big] \\
%         &\quad - S \mathbf{M} \big[ \gamma(\beta^{n+1} - \beta^*) + \delta(\beta^* - \beta^n) \big],
%     \end{aligned}
% \end{equation}
% where the scheme parameters are $\gamma = 1 - 1/\sqrt{2}$ and $\delta = 1 - 1/(2\gamma)$. 
\end{itemize}

Figure \ref{fig:ac_1D_convergence} presents the temporal 
convergence results. 
Both schemes achieve their expected theoretical orders, 
indicating that the Galerkin projection preserves 
the temporal accuracy of the underlying integrators. 
As in the heat equation example, 
the IMEX-RK2 error eventually reaches a plateau, 
corresponding to the spatial approximation limit 
of the neural basis. 
In the following experiments, we adopt IMEX-RK2 
due to its higher temporal accuracy.
\begin{figure}[htbp]
    \centering
    \includegraphics[width=\textwidth]{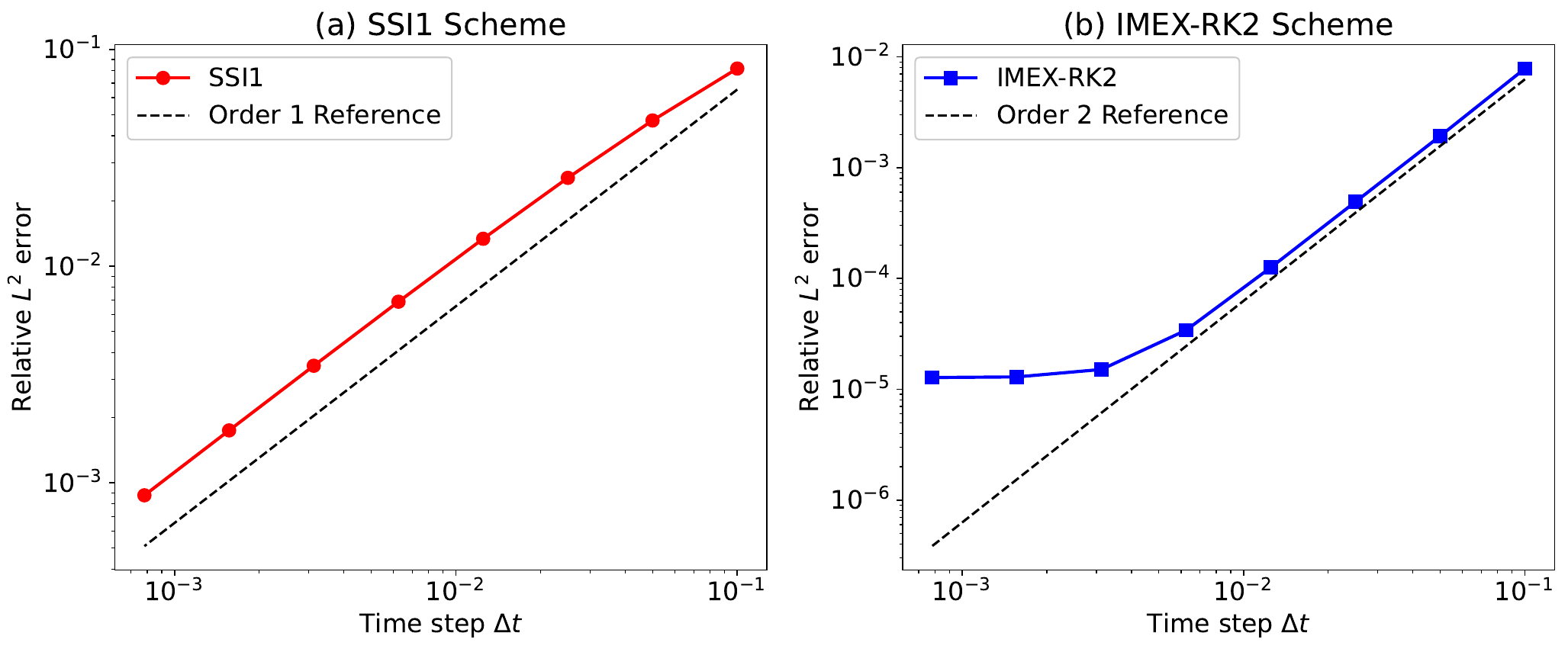}
    \caption{Temporal convergence test for the 1D Allen--Cahn equation (Example \ref{eg:1d_AC}) comparing the SSI1 and IMEX-RK2 schemes with adaptive DNN Galerkin method.}
    \label{fig:ac_1D_convergence}
\end{figure}

Figure \ref{fig:ac_1d_results} compares the DNN-Galerkin solution, 
a standard PINN, and a spectral reference solution. 
Panels (a) and (b) show the solution profiles at $t=0.2$ and $t=1.0$. 
While the PINN captures the overall phase transition, 
visible discrepancies appear near the sharp interface. 
The ADNN-G solution remains in close agreement 
with the spectral reference throughout the evolution.

The relative $L^2$ error evolution in 
Fig.~\ref{fig:ac_1d_results}(c) 
demonstrates that the DNN-G method 
maintains consistently lower errors over time. 
More importantly, 
Fig.~\ref{fig:ac_1d_results}(d) verifies 
the discrete energy dissipation property. 
The DNN-G scheme exhibits monotone decay 
of the free energy, 
consistent with the gradient flow structure. 
In contrast, the PINN solution shows 
noticeable energy gaps, 
reflecting the absence of structural constraints 
in residual-based training.

\begin{figure}[htbp]
    \centering
    \includegraphics[width=\textwidth]{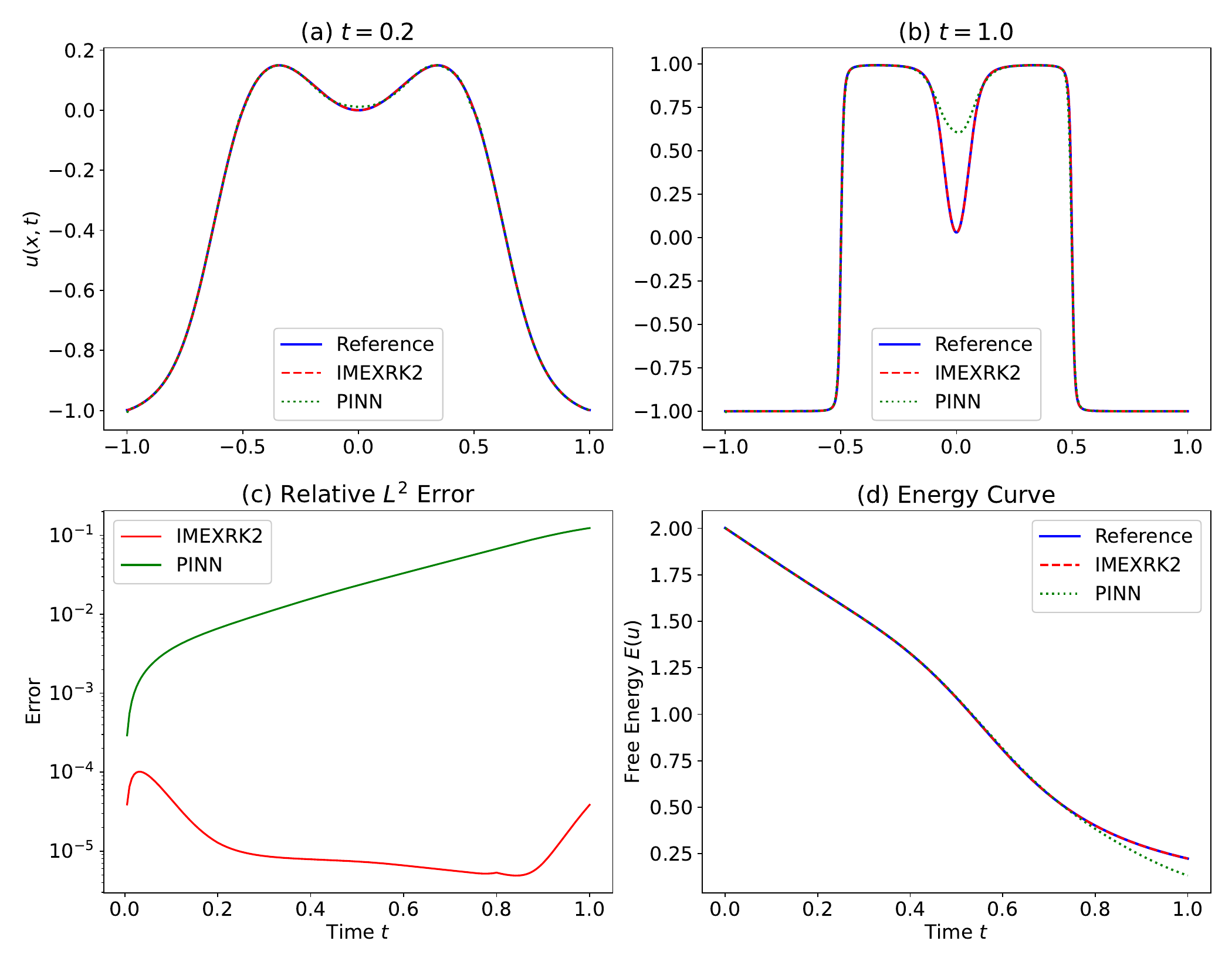}
    \caption{Numerical results for the 1D Allen--Cahn equation (Example \ref{eg:1d_AC}). (a)--(b) Comparison of solution profiles $u(x,t)$. (c) Temporal evolution of the relative $L^2$ error. (d) Free energy dissipation $E(u)$ curves.}
    \label{fig:ac_1d_results}
\end{figure}

\subsubsection{2D Case}

\begin{example}[The evolution of two circular bubbles]
   We consider the two-dimensional Allen--Cahn equation on 
$\Omega = [-0.5, 0.5]^2$:
\begin{equation}
    u_t = \epsilon^2 \Delta u - \kappa f(u), \quad (x,y) \in \Omega, \quad t \in [0, 1],
\end{equation}
where $\epsilon = 0.05$, $\kappa = 5$. 
The initial condition describes two separate circular bubbles:
\begin{equation}
    u_0(x,y) = \max \left\{ \tanh\left(\frac{R - d_1}{\epsilon}\right), \tanh\left(\frac{R - d_2}{\epsilon}\right) \right\},
\end{equation}
where $R=0.15$ is the radius, and $d_{1,2} = \sqrt{(x \pm 0.21)^2 + y^2}$.
\label{eg:ac_2d_circle}
\end{example}
Figure~\ref{fig:ac_2d_evolution} illustrates the evolution process of the proposed 
adaptive DNN Galerkin method (ADNN-G), the standard PINN, 
and the traditional spectral method under equivalent degrees of freedom (DoF). 
Driven by mean curvature, the two nearby interfaces gradually coalesce, 
which serves as a stringent test for accurately resolving high-curvature dynamics.

With $m=1024$ adaptive neural basis functions,
the proposed ADNN-G method accurately captures the merging process
and closely matches the high-resolution spectral reference solution.
Among all tested methods, ADNN-G with $\text{DoF}=1024$ achieves the highest accuracy
and reproduces the energy dissipation curve almost identically to the reference solution.

For a fair comparison under similar degrees of freedom,
we further compare ADNN-G with $\text{DoF}=64$ against the spectral method with
$\text{DoF}=32^2=1024$.
Although the two solutions appear visually similar in the snapshots,
the ADNN-G solution aligns better with the energy dissipation curve.
This indicates that the adaptive neural basis functions provide
a more efficient representation of the evolving interface dynamics.
In contrast, both the standard PINN and the spectral method with
$\text{DoF}=8^2$ fail to capture the interface coalescence within the simulated time horizon,
indicating insufficient resolution to represent the strongly nonlinear dynamics.

The quantitative comparison in Fig.~\ref{fig:ac_2d_analysis} further confirms these observations.
As shown in Fig.~\ref{fig:ac_2d_analysis}(a), ADNN-G with $\text{DoF}=1024$
achieves the smallest relative $L^2$ errors among all methods.
Meanwhile, ADNN-G with only $\text{DoF}=64$ attains an accuracy comparable to
the spectral method with $\text{DoF}=32^2$, despite using significantly fewer degrees of freedom.
This clearly demonstrates a super-approximation effect:
while fixed trigonometric bases require substantially higher resolution
to resolve sharp interfaces, the adaptive neural basis functions
automatically concentrate in high-gradient regions.

Additionally, the energy dissipation curves in
Fig.~\ref{fig:ac_2d_analysis}(b) show that both ADNN-G schemes follow
the reference thermodynamic path more faithfully than the competing methods,
highlighting the structure-preserving property of the proposed framework.

\begin{figure}[htbp]
    \centering
    \includegraphics[width=\textwidth]{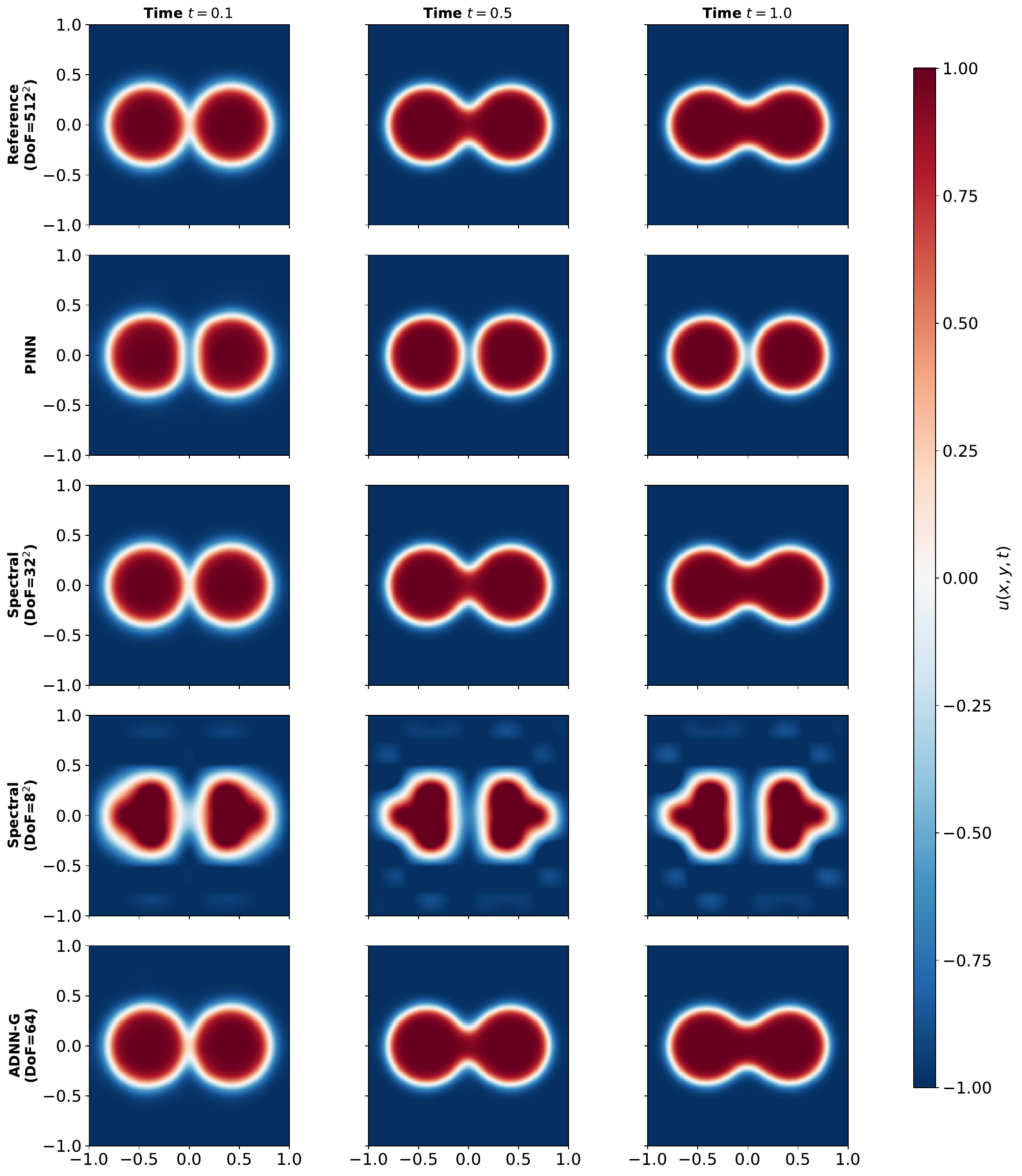}
    \caption{Snapshots of the phase-field $u(x,y,t)$ evolution for the merging process of two
     circular interfaces (Example \ref{eg:ac_2d_circle}). 
     Rows from top to bottom: Reference spectral solution (DoF = $512^2$), standard PINN, 
     spectral solution (DoF = $32^2$ and $8^2$),
     and the proposed adaptive DNN Galerkin (ADNN-G, DoF = 64) scheme at $t=0.1, 0.5$, and $1.0$. }
    \label{fig:ac_2d_evolution}
\end{figure}
\begin{figure}[htbp]
    \centering
    \includegraphics[width=\textwidth]{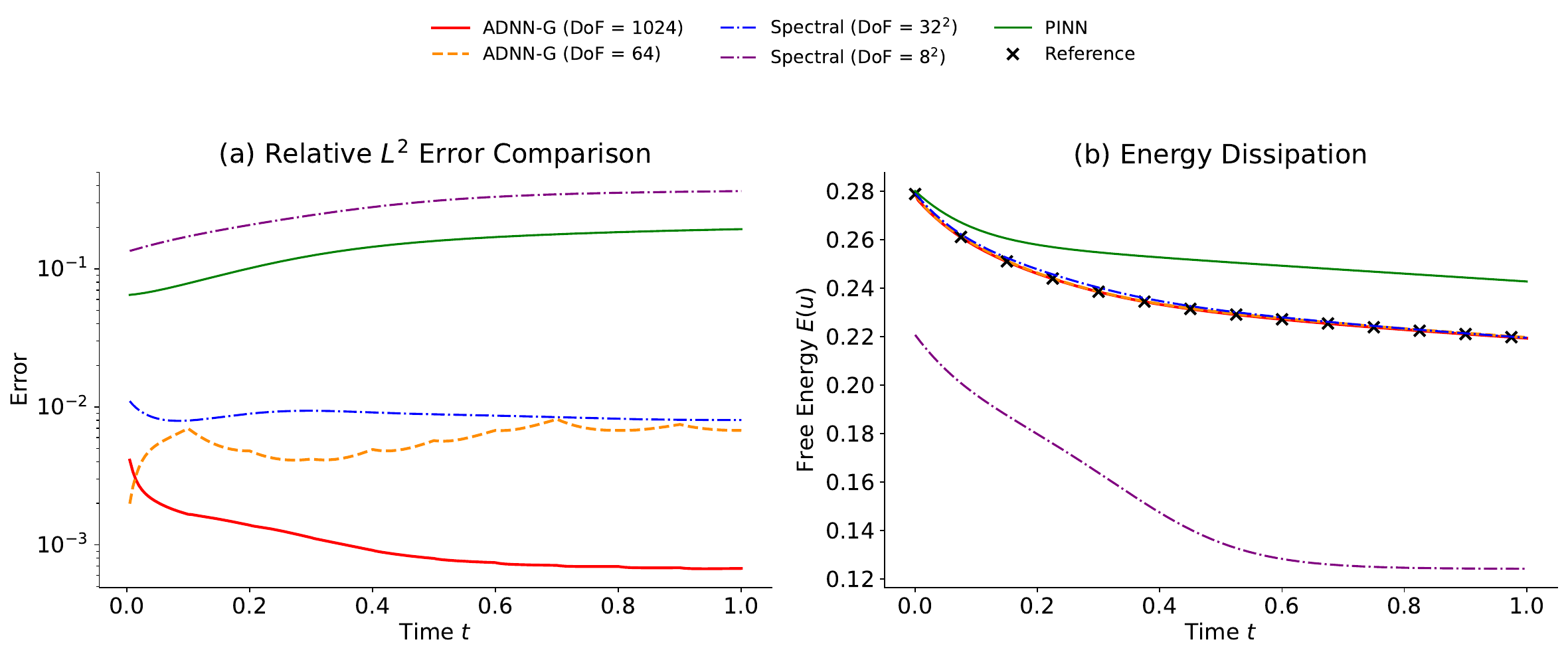}
    \caption{Quantitative comparison for the 2D Allen--Cahn equation (Example \ref{eg:ac_2d_circle}). (a) Relative $L^2$ error evolution (excluding $t=0$ to highlight dynamical differences). (b) Free energy dissipation $E(u)$ comparison.}
    \label{fig:ac_2d_analysis}
\end{figure}

%\subsubsection{2D Case: Non-uniform Curvature and Morphological Relaxation}

\begin{example}[Non-uniform curvature-driven motion]
    To further evaluate the ability of DNN-G to resolve phase-field dynamics with spatially varying curvature, 
we consider a star-shaped interface defined in polar coordinates $(r,\theta)$:
    \begin{equation}
        u_0(r, \theta) = \tanh \left( \frac{R_0 + \delta \cos(k\theta) - r}{\sqrt{2}\epsilon} \right),
    \end{equation}
    where $R_0 = 0.25$, $\delta = 0.05$, and $k=5$.  The parameters $\epsilon = 0.05$ and $\kappa = 5$ 
    are the same as those in Example \ref{eg:ac_2d_circle}.
    \label{eg:ac_2d_star}
\end{example}

The evolution of the star-shaped interface is presented in Fig.~\ref{fig:ac_2d_star_evolution}. 
As illustrated in the snapshots, the standard PINN exhibits a non-physical failure, while 
the RF-DNN-G scheme also struggles to capture the correct dynamics, with the solution remaining close to the initial condition.
In contrast, both ADNN-G and SFLI-DNN-G accurately reproduce the morphological evolution, 
capturing the relaxation from the star-shaped interface toward the circular equilibrium.

\begin{figure}[htbp]
    \centering
    \includegraphics[width=0.99\textwidth]{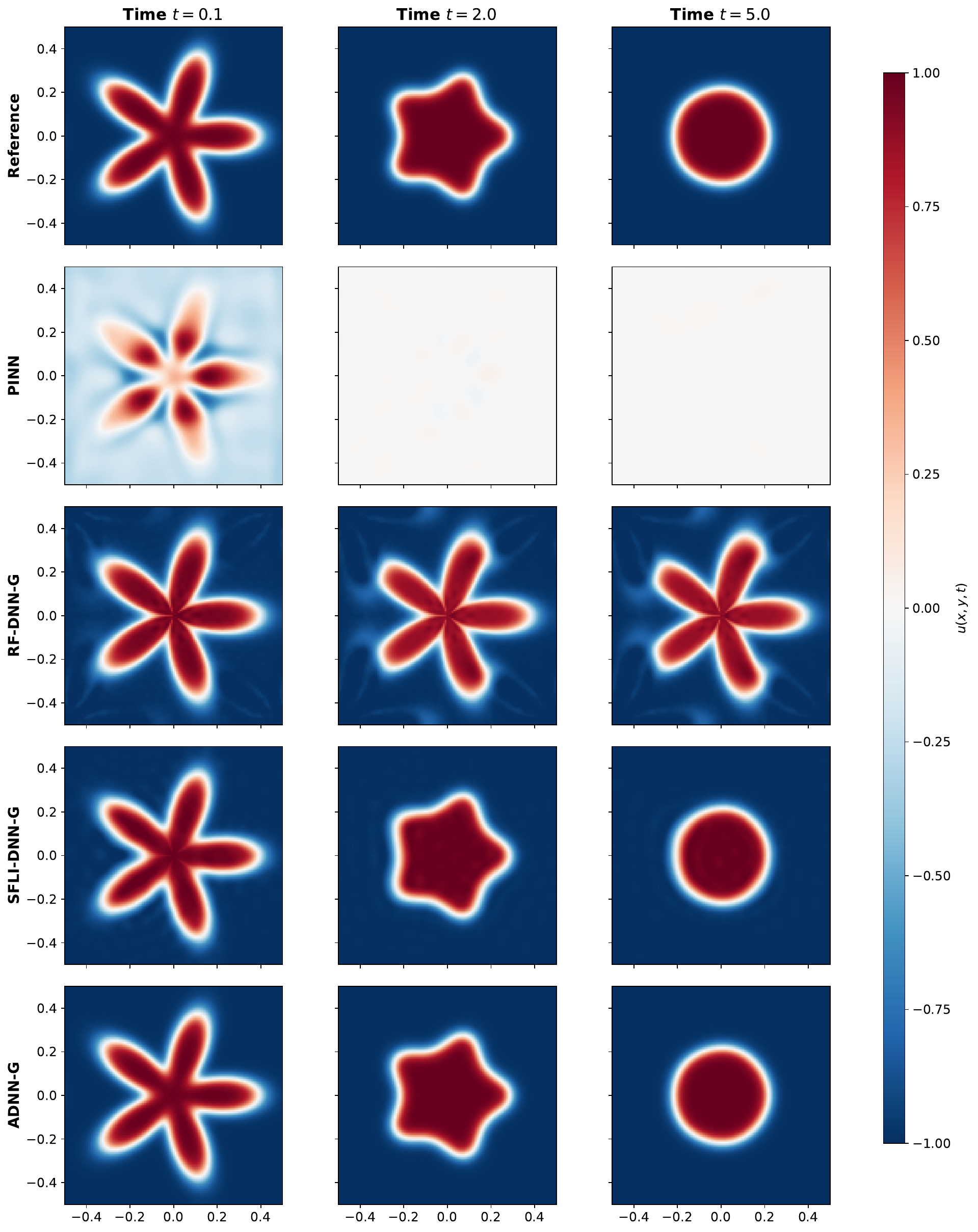}
    \caption{Snapshots of the phase-field evolution for a star-shaped interface (Example \ref{eg:ac_2d_star})  at $t=0.1, 2.0,$ and $5.0$. 
    From top to bottom: Reference spectral solution, standard PINN, random feature based DNN-Galerkin scheme (RF-DNN-G), SFLI based DNN-Galerkin scheme (SFLI-DNN-G), 
    and the adaptive DNN-Galerkin scheme (ADNN-G).}
    \label{fig:ac_2d_star_evolution}
\end{figure}

Quantitative evaluations in Fig.~\ref{fig:ac_2d_star_analysis} corroborate these observations. 
As shown in Fig.~\ref{fig:ac_2d_star_analysis}(a), the relative $L^2$ errors of SFLI-DNN-G and ADNN-G remain 
significantly lower than those of PINN and RF-DNN-G throughout the simulation. 
The free energy evolution in Fig.~\ref{fig:ac_2d_star_analysis}(b) further illustrates the 
structure-preserving property of the Galerkin formulation. 
Both SFLI-DNN-G and ADNN-G closely follow the reference energy dissipation curve. 
Although RF-DNN-G yields poor predictive accuracy, its Galerkin structure ensures that the energy remains non-increasing. 
In contrast, the PINN solution exhibits both large errors and non-physical energy growth, 
highlighting the difficulty of directly minimizing residual losses for stiff gradient flows.
\begin{figure}[htbp]
    \centering
    \includegraphics[width=\textwidth]{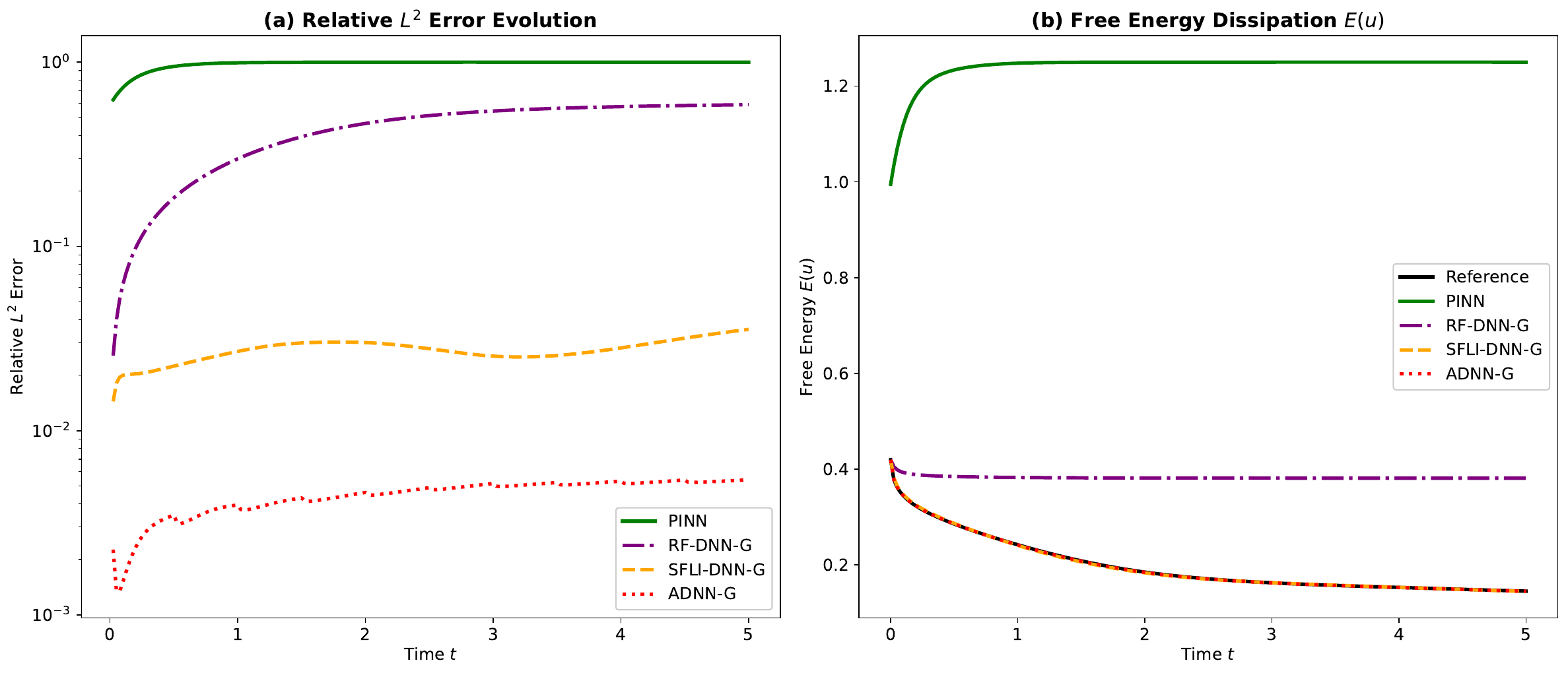}
    \caption{Quantitative comparison for the star-shaped interface relaxation (Example \ref{eg:ac_2d_star}). (a) Relative $L^2$ error evolution. (b) Free energy dissipation $E(u)$.}
    \label{fig:ac_2d_star_analysis}
\end{figure}

A key insight from this example is the essential role of basis adaptivity. 
The poor performance of RF-DNN-G indicates that a structure-preserving Galerkin projection alone cannot compensate 
for an inadequate trial space. In contrast, the superior accuracy of ADNN-G compared with SFLI-DNN-G highlights a 
synergistic effect: although direct PINN optimization fails to recover the correct macroscopic dynamics, 
the training process implicitly extracts informative nonlinear features that align the neural basis functions with the underlying solution manifold. 
The subsequent Galerkin projection then exploits these learned bases to accurately recover the dynamics. 
This observation suggests that while residual minimization may struggle as a standalone solver for stiff gradient flows, 
it can serve as an effective mechanism for constructing an expressive low-dimensional trial space for Galerkin-based solvers.

We further examine the proposed DNN Galerkin framework in a more demanding regime: 
coarsening dynamics initiated from a random state. 
This setup involves multiple simultaneous coalescence events and rapidly evolving 
multi-scale structures, posing a stringent test for neural-network-based PDE solvers. 
To the best of our knowledge, existing neural-network approaches have not demonstrated 
stable and physically consistent simulations of the Allen--Cahn equation starting from 
fully random initial data, making this scenario particularly challenging.

\begin{example}[Coarsening from a random initial condition]
   The governing equation and parameters are the same as in 
Example~\ref{eg:ac_2d_circle}.
    The initial condition is sampled from a uniform distribution:
    \begin{equation}
        u_0(x,y) \sim \mathcal{U}(-1,1), \quad (x,y) \in \Omega.
    \end{equation}
Since a purely random field lacks spatial regularity, 
we project $u_0$ onto a finite-dimensional spectral basis to obtain a smooth but highly oscillatory initial configuration.

        \label{eg:ac_2d_random}
\end{example}

\begin{figure}[htbp]
    \centering
    \includegraphics[width=0.95\textwidth]{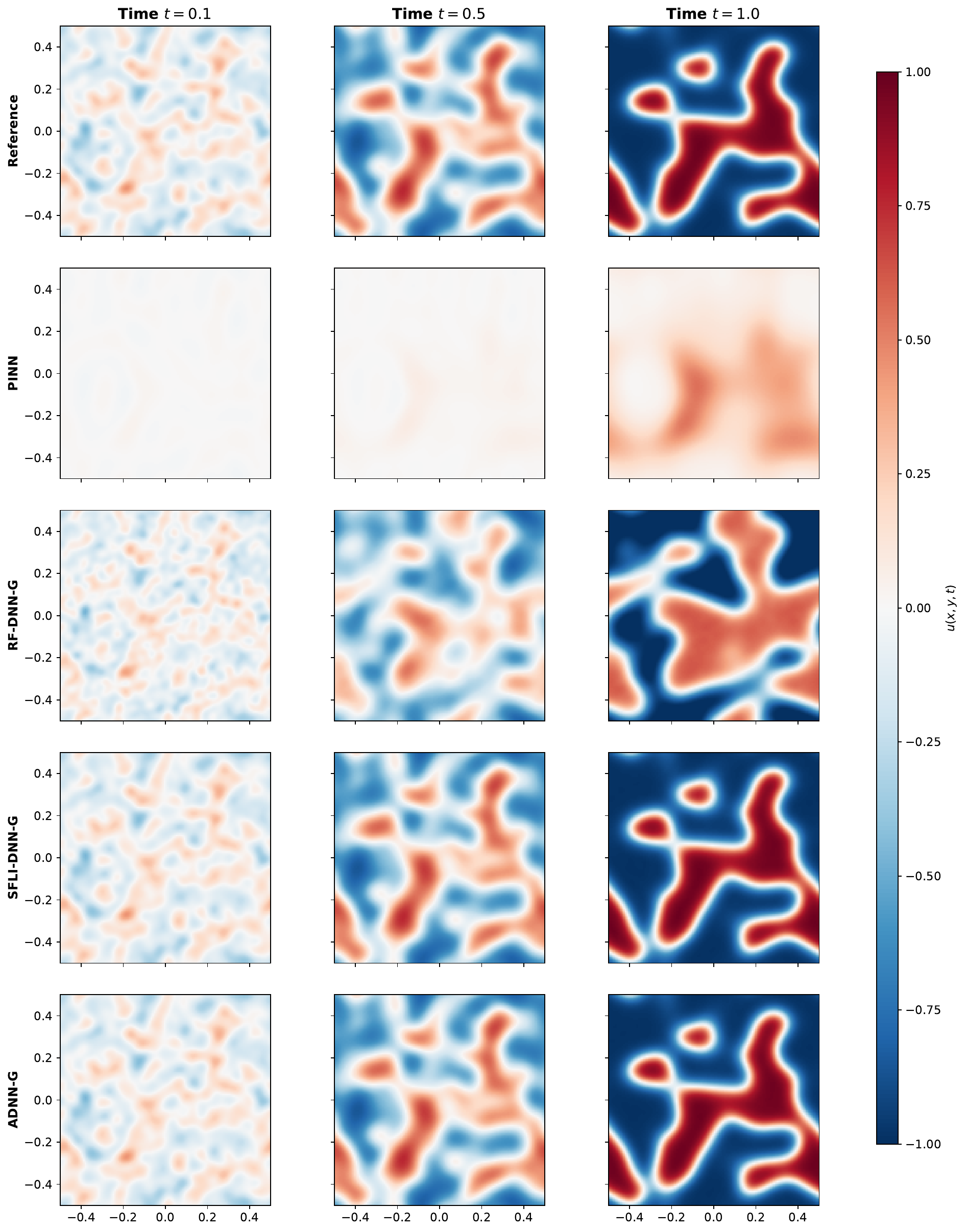}
    \caption{Snapshots of the phase-field $u(x,y,t)$ evolution for the random initial condition 
    (Example \ref{eg:ac_2d_random}) at $t=0.1$, $0.5$, and $1.0$. 
    From top to bottom: Reference spectral solution, standard PINN, random feature based DNN-Galerkin scheme (RF-DNN-G), SFLI based DNN-Galerkin scheme (SFLI-DNN-G), 
    and the adaptive DNN-Galerkin scheme (ADNN-G).}
    \label{fig:ac_2d_rand_evolution}
\end{figure}
\begin{figure}[htbp]
    \centering
    \includegraphics[width=0.75\textwidth]{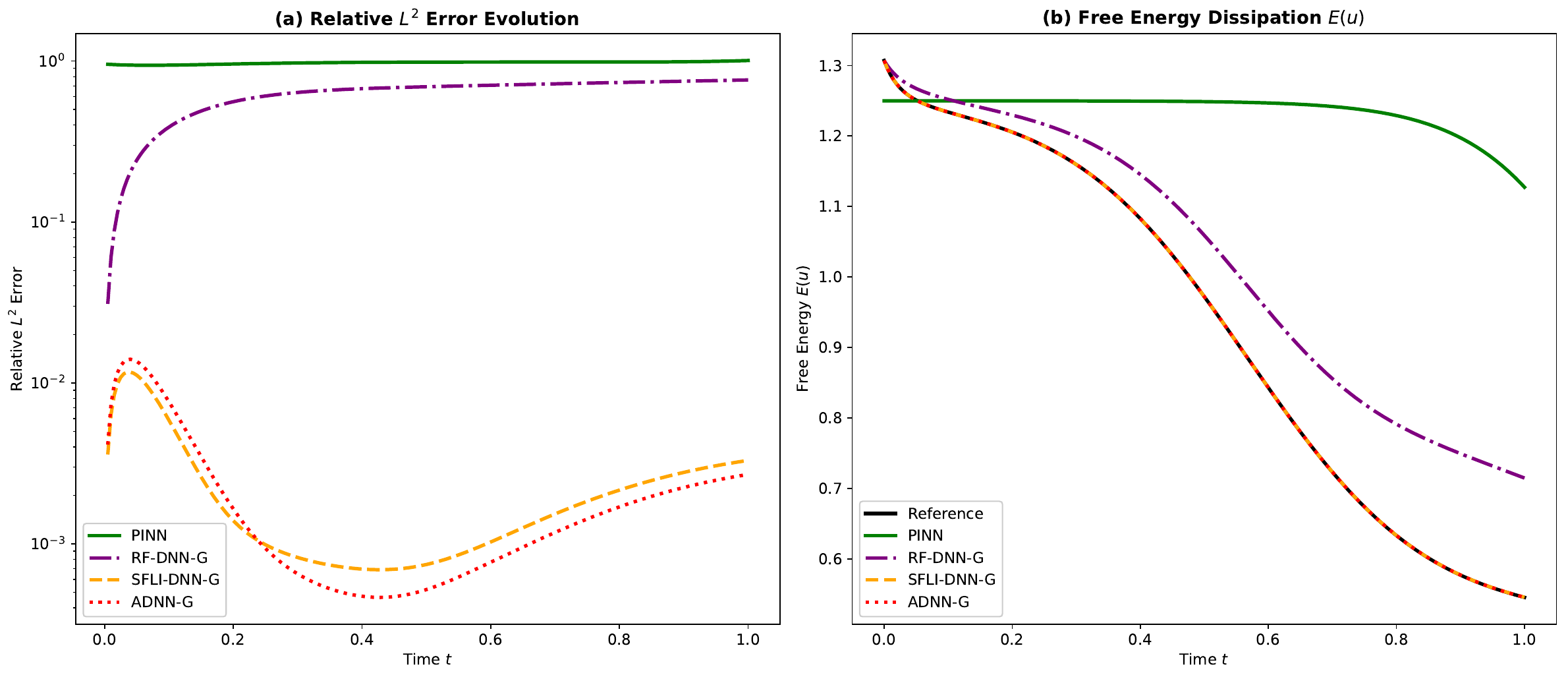}
    \caption{Quantitative comparison for the coarsening dynamics from a random initial condition (Example \ref{eg:ac_2d_random}).
     (a) Relative $L^2$ error evolution. (b) Free energy dissipation $E(u)$ comparison.}
    \label{fig:ac_2d_rand_energy}
\end{figure}

The numerical results in 
Fig.~\ref{fig:ac_2d_rand_evolution} and 
Fig.~\ref{fig:ac_2d_rand_energy} 
demonstrate a clear contrast between optimization-based and Galerkin-based approaches. 
The standard PINN struggles to represent the rapidly evolving multi-scale structures, 
leading to noticeable deviations from the reference solution even at early times. 
The RF-DNN-G scheme exhibits intermediate performance: although the Galerkin projection 
improves the qualitative topological evolution compared with PINN, its fixed random trial 
space lacks sufficient expressivity to accurately resolve the interface dynamics. 
In contrast, both SFLI-DNN-G and ADNN-G successfully reproduce the coarsening process, 
capturing successive interface mergers while maintaining consistent energy dissipation.

These results highlight the importance of constructing an expressive trial space 
for complex coarsening dynamics. 
Compared with the inadequate expressivity of random features, the neural basis 
generated by SFLI already provides an effective static representation of the 
solution manifold. 

More importantly, the results emphasize that the structure-preserving Galerkin 
formulation plays a central role in ensuring physical consistency. 
By embedding the gradient-flow structure into the numerical scheme, the 
DNN-Galerkin framework guarantees stable energy dissipation even during 
highly nonlinear, multi-scale evolution, whereas direct residual minimization 
alone struggles to maintain such physical properties.

\subsection{Cahn--Hilliard equation}

We further evaluate the proposed framework on the Cahn--Hilliard equation,
which describes the $H^{-1}$ gradient flow of the Ginzburg--Landau free energy
\cite{cahn_freeenergynonuniform_1958}.
Compared with the Allen--Cahn equation, the Cahn--Hilliard model involves
fourth-order spatial derivatives and mass conservation,
resulting in significantly stronger numerical stiffness.
This example therefore provides a stringent test for neural-network-based
solvers for high-order phase-field dynamics.

\begin{example}[Cahn--Hilliard equation]
Consider the one-dimensional Cahn--Hilliard equation
\begin{equation}
u_t = -\epsilon^2 u_{xxxx} + (u^3-u)_{xx},
\quad x \in [-1,1], \quad t \in [0,1].
\end{equation}
The initial condition is
$
u_0(x)=\cos(\pi x).
$
To examine both mild and stiff regimes,
two interfacial width parameters are considered:
$\epsilon=0.1$ and $\epsilon=0.01$.
\label{eg:ch_1d}
\end{example}

For time integration, we employ the similar IMEX-RK2 scheme used in the
Allen--Cahn example, which can be found in \cite{fu_energydiminishingimplicitexplicit_2024}.

Figure~\ref{fig:ch_1d_contour} shows the spatiotemporal evolution of the
solution using $x$-$t$ contour plots.
For the mild regime ($\epsilon=0.1$), both PINN and ADNN-G capture the
overall phase separation pattern.
Fig.~\ref{fig:ch_1d_quantitative} further indicates that ADNN-G maintains
smaller relative $L^2$ errors throughout the simulation.

The difference becomes more significant in the stiff regime
($\epsilon=0.01$), where the interface width is much smaller and the
dynamics are highly stiff.
In this case, the standard PINN fails to accurately resolve the
evolution due to the severe stiffness induced by the fourth-order
spatial derivatives, leading to large approximation errors and
non-physical energy behavior.
In contrast, the ADNN-G scheme remains stable and accurately reproduces
the phase separation dynamics.

This robustness stems from the weak Galerkin formulation, which
transfers the high-order derivatives onto the basis functions through
integration by parts.
As a result, the regularity requirements imposed on the neural network
representation are significantly relaxed.

The quantitative results in Fig.~\ref{fig:ch_1d_quantitative}
further demonstrate the advantage of the proposed method.
While the PINN solution in the stiff regime exhibits energy stagnation
and noticeable fluctuations, the ADNN-G scheme preserves the correct
monotonic free-energy dissipation and closely follows the reference
solution throughout the evolution.

\begin{figure}[htbp]
    \centering
    \includegraphics[width=\textwidth]{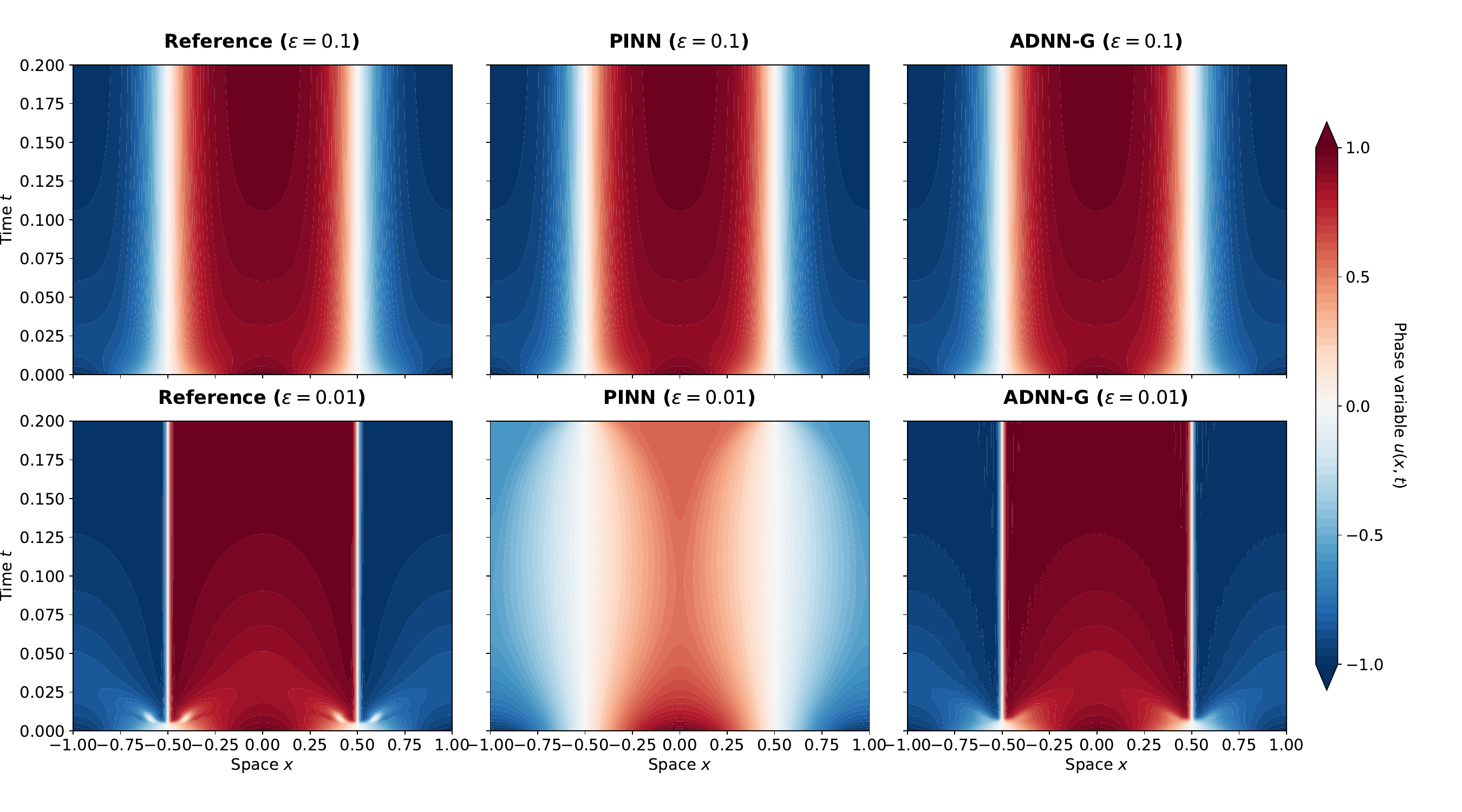}
 \caption{Spatiotemporal evolution of the phase variable $u(x,t)$ 
 for the Cahn--Hilliard dynamics (Example \ref{eg:ch_1d}) 
with $\epsilon=0.1$ (top row) and $\epsilon=0.01$ (bottom row).
From left to right: Reference spectral solution, standard PINN, and the proposed ADNN-G scheme.}
    \label{fig:ch_1d_contour}
\end{figure}
\begin{figure}[htbp]
    \centering
    \includegraphics[width=0.8\textwidth]{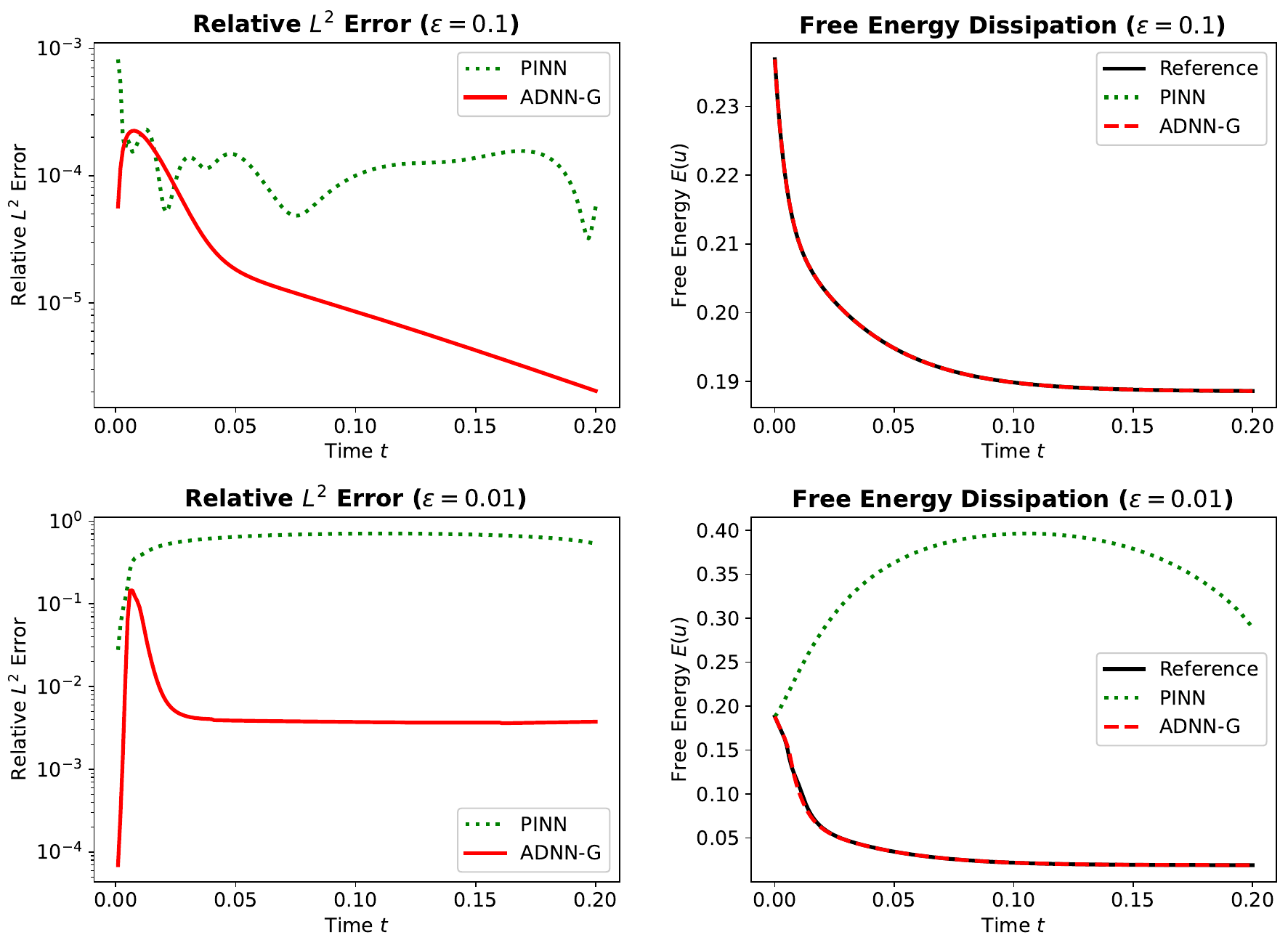}
\caption{Relative $L^2$ error and free-energy dissipation for the Cahn--Hilliard dynamics (Example \ref{eg:ch_1d}) 
with $\epsilon=0.1$ (top row) and $\epsilon=0.01$ (bottom row).}
    \label{fig:ch_1d_quantitative}
\end{figure}

\section{Conclusion}

In this work, we proposed a structure-preserving deep neural network Galerkin framework for nonlinear dissipative partial differential equations. By embedding neural networks into a Galerkin projection formulation, the method preserves the variational structure of gradient flows and guarantees discrete energy dissipation throughout the numerical evolution.

Extensive numerical experiments demonstrate that the proposed approach provides improved accuracy and robustness compared with optimization-based neural PDE solvers, particularly for high-dimensional and multi-scale dynamics. Under comparable degrees of freedom, the DNN-Galerkin method can also outperform classical spectral discretizations, indicating that neural trial spaces offer strong approximation capability for complex evolving structures.

An important observation from this study is that the effectiveness of neural PDE solvers is largely determined by the quality of the induced trial space rather than by the optimization procedure itself. Even when direct neural training fails to converge to a satisfactory solution, properly constructed neural basis functions can still provide efficient approximation spaces when combined with a structure-preserving Galerkin projection.

These findings suggest a different perspective on neural approaches to scientific computing: deep neural networks can serve as adaptive basis generators within structure-aware numerical discretization frameworks. Future work will extend the proposed methodology to more challenging dissipative systems, including stiff gradient flows, higher-dimensional problems, and kinetic equations, further exploring neural representations as flexible approximation spaces for large-scale PDE simulations.

\section*{Acknowledgments}
This work is supported by the National Science Foundation of China (No.12271240, 12426312), the fund of the Guangdong Provincial Key Laboratory of Computational Science and Material Design, China (No.2019B030301001), the Shenzhen Natural Science Fund (RCJC20210609103819018), and the Zhuhai Innovation and Entrepreneurship Team Project (2120004000498).

%\cite{CiCP-28-2139}
%% 参考文献配置 
\bibliographystyle{elsarticle-num} 
\bibliography{ref_up}

\end{document}